\documentclass[10pt]{amsart}
\usepackage{graphics}
\usepackage[pdftex]{hyperref}
\usepackage{amssymb,amsmath,amsfonts,stackengine}
\usepackage{bm}
\usepackage{mathtools}

\DeclarePairedDelimiter\floor{\lfloor}{\rfloor}
\usepackage{datetime}
\usepackage{color}
\usepackage{ulem}
\usepackage{lipsum}
\usepackage{amsfonts}
\usepackage{enumerate}
\usepackage[at]{easylist}
\usepackage{cite}
\usepackage{soul}
\usepackage{tikz}
\usetikzlibrary{matrix}
\usepackage[pdftex]{hyperref}
\RequirePackage{amscd}
\RequirePackage{epic}
\RequirePackage[all]{xy}
\RequirePackage{url}
\RequirePackage[shortlabels]{enumitem}
\usepackage{empheq}
\usepackage{epsfig}
\usepackage{lineno} 
 \usepackage{perpage}
 \MakePerPage{linenumbers}
\usepackage[english]{babel}

\usepackage[utf8]{inputenc}
\RequirePackage{amscd}
\RequirePackage{epic}
\RequirePackage{eepic}
\RequirePackage[all]{xy}
\RequirePackage{url}
\usepackage{empheq}
\usepackage{nomencl}
\usepackage{epsfig}
\usepackage{graphicx}

\newcommand{\normcl}[3]{|#1|_{C^{#2}(#3)}}
\newcommand{\normf}[3]{\left|\left|#1\right|\right|_{#2,#3}}

\newcommand{\teta}{\theta}
\newcommand{\eps}{\varepsilon}

\newcommand{\smooth}[1]{\mathtt{#1}_s}

\newcommand\ledot{\mathrel{\ensurestackMath{%
			\stackengine{-.5ex}{\lessdot}{-}{U}{c}{F}{F}{S}}}} 
		
\newcommand\gedot{\mathrel{\ensurestackMath{%
			\stackengine{-.5ex}{\gtrdot}{-}{U}{c}{F}{F}{S}}}}

\newcommand{\comment}[1]{}
   
    \newcommand{\set}[1]{{\left\{#1\right\}}}
\newcommand{\pa}[1]{{\left(#1\right)}}

\newcommand{\abs}[1]{{\left|#1\right|}}
\newcommand{\norm}[1]{{\left |#1\right |}}

\newcommand{\A}{\mathbb{A}}

\newcommand{\T}{\mathbb{T}}
\newcommand{\Z}{\mathbb{Z}}
\newcommand{\R}{\mathbb{R}}
\newcommand{\C}{\mathbb{C}}

\renewcommand{\Re}{\operatorname{Re}}
\renewcommand{\Im}{\operatorname{Im}}

\newcommand{\co}[1]{\textit{#1}}
\newcommand{\gr}[1]{\textbf{#1}}

\newcommand{\id}{\operatorname{id}}

\usepackage{amsthm}

\RequirePackage{url}
\newtheorem{prop}{Proposition}[section]
    \newtheorem{thm}{Theorem}[section]
    \newtheorem*{thm*}{Theorem}
    \newtheorem*{cor*}{Corollary}

    \newtheorem{lemma}{Lemma}
    \theoremstyle{remark}
     \newtheorem{ex}{Example}
\newtheorem{rmk}{Remark}[section]
\theoremstyle{definition}
\newtheorem{defn}{Definition}

\numberwithin{equation}{section}
\numberwithin{thm}{section}
\numberwithin{defn}{section}
\numberwithin{prop}{section}
\numberwithin{cor}{section}
\numberwithin{lemma}{section}
\numberwithin{rmk}{section}

\newcommand{\s}{{\sigma}}

\newcommand{\om}{{\omega}}

\newcommand{\F}{{\mathbb F}}

\newcommand{\N}{{\mathbb N}}

\newcommand{\cA}{{\mathcal A}}
\newcommand{\cB}{{\mathcal B}}

\newcommand{\cD}{{\mathcal D}}

\newcommand{\cI}{{\mathcal I}}

\newcommand{\cM}{{\mathcal M}}
\newcommand{\cN}{{\mathcal N}}

\newcommand{\cR}{{\mathcal R}}

\newcommand{\cZ}{{\mathcal Z}}

\newcommand{\fk}{{\mathfrak{k}}}

\newcommand{\tf}{{\mathtt{f}}}
\newcommand{\tg}{{\mathtt{g}}}

\newcommand{\tC}{{\mathtt{C}}}

\newcommand{\tH}{{\mathtt{H}}}

\newcommand{\tI}{{\mathtt{I}}}

\newcommand{\tR}{{\mathtt{R}}}

\newcommand{\bM}{{\bf M}}

\usepackage{bm}

\newcommand{\al}{{\alpha}}
\newcommand{\alg}{{\bm{\alpha}}}

\renewcommand{\d}{\partial}

\newcommand\norma[1]{\left\lVert#1\right\rVert}

\newcommand{\im}{{\rm i}}

\newcommand{\nnorm}[1]{{\left\vert\kern-0.25ex\left\vert\kern-0.25ex\left\vert #1 
    \right\vert\kern-0.25ex\right\vert\kern-0.25ex\right\vert}}

\newcommand{\La}[1]{\Lambda^{\pa{#1}}}

\newcommand{\bal}{{\bm \al}}

\usepackage{framed,enumitem}

\newcommand{\pallara}{B_\infty(0,R)}
\newcommand{\pallaradue}{B_\infty(0,2R)}

\newcommand{\pallaro}{\pallara \times \T^n}

\newcommand{\jack}{\tC_{\mathtt{ J}}}

\newcommand{\jess}{\tC_A}

\newcommand{\santi}{\tC_B}
\newcommand{\jeanpierre}{\tC_{\mathtt{ F}}}

\newcommand{\ciccio}[1]{\cD_{#1}}

\newcounter{paraga}[subsection]
\renewcommand{\theparaga}{{\bf\arabic{paraga}.}}
\newcommand{\paraga}{\medskip \addtocounter{paraga}{1} 
\noindent{\theparaga\ } }
\def\RR{{\bf R}}
\def\TT{{\bf T}}
\def\EE{{\bf E}}
\def\dsp{\displaystyle}
\newcommand{\norms}[1]{\vert #1\vert}
\newcommand{\dr}{\partial}
\newcommand{\alp}{\alpha}
\def\ov{\overline}
\def\ha{\widehat}
\def\beq{\begin{equation}}
\def\eeq{\end{equation}}
\def\jS{{\mathcal S}}
\def\f{{\bf f}}
\def\F{{\bf F}}
\def\bT{{\bf T}}
\def\bW{{\bf W}}
\def\dst{\displaystyle}
\def\dsp{\displaystyle}
\def\setm{\setminus}

\def\th{\theta}

\def\La{\Lambda}
\def\vect{\rm Vect\,}
\def\demi{\frac{1}{2}}
\def\cB{{\mathcal B}}

\def\La{\Lambda}

\usepackage[pdftex]{hyperref}

\begin{document}

 \title[]{{Analytic Smoothing and Nekhoroshev estimates for H\"older steep Hamiltonians}}
\author{Santiago  Barbieri}
\address{Université Paris-Saclay and Università degli Studi Roma Tre}
\email{santiago.barbieri@universite-paris-saclay.fr }

\author{Jean-Pierre Marco}
\address{Sorbonne Université }
\email{jean-pierre.marco@imj-prg.fr }

\author{Jessica Elisa Massetti}
\address{Università degli Studi Roma Tre}
\email{jessicaelisa.massetti@uniroma3.it}

\begin{abstract}
In this paper we prove the first result of  Nekhoroshev stability for steep Hamiltonians in H\"older class. 
Our new approach combines the classical theory of normal forms in analytic category with an improved smoothing procedure to approximate an H\"older Hamiltonian with an analytic one.
 It is only for the sake of clarity that we consider the (difficult) case of H\"older perturbations of an analytic
integrable Hamiltonian, but our method is flexible enough to work in many other functional classes,
including the Gevrey one. The stability exponents can be taken 
to be $(\ell-1)/(2n\alg_1...\alg_{n-2})+1/2$ for the time of stability and  $1/(2n\alg_1...\alg_{n-1})$ for the radius of stability, 
$n$ being the dimension, $\ell >n+1$ being the regularity and the $\bal_i$'s being the indices of steepness.    Crucial to obtain the exponents above is a new non-standard estimate on the Fourier norm of the smoothed function.  As a byproduct we improve the stability exponents in the $C^k$ class,  with integer $k$.  
\end{abstract} 
\maketitle
\setcounter{tocdepth}{1} 
\tableofcontents

\section{Introduction and main results}



\paraga The main goal of this work is to introduce a unified way for proving ``long time stability'' of the action variables for perturbations of completely
integrable Hamiltonian systems which belong to a large class of function spaces. We will limit ourselves here to H\"older  
perturbations of analytic systems, but our method is flexible enough to be adapted to many other settings\footnote{Assuming that the unperturbed 
system is analytic is just a matter of simplification.  }.


The effective stability theory for nearly-integrable hamiltonian systems was initiated by the pioneering work of J.E. Littlewood 
\cite{Littlewood_1959} and 
reached a first main achievement in the seventies with the work of N.N. Nekhoroshev  \cite{Nekhoroshev_1977}; it was then developed by many authors. 
The usual setting is that of Hamiltonian systems  of the form
\begin{equation}\label{eq:hampert}
H(I,\teta)=h(I)+ f(I,\theta),
\end{equation}
where $(I,\theta)\in\R^n\times \T^n$ are the action-angle variables and $f$ is small with respect to $h$. In Nekhoroshev's work the Hamiltonian $H$
is analytic and $h$ satisfies a steepness condition (see definition \ref{def steep} below). The theory has been then developed 
in various settings:  $H$ can be assumed to be Gevrey (which includes the analytic case) or $C^k$ with $k\geq 2$ and integer, 
while $h$ can be assumed to be convex or quasi-convex (see for example  \cite{Marco_Sauzin_2003} or \cite{Bounemoura_2010}) 

The norm of $f$, relative to the function space at hand, is denoted by $\eps$.
For systems as (\ref{eq:hampert}), the previous results assert that
the action variables are confined in a ball of radius $\RR(\eps)$ centered at the initial action
during a time $\TT(\eps)$, provided that $\eps$ is smaller that some threshold $\EE$.
We say that $\RR(\eps)$ is the {\it confinement radius}, $\TT(\eps)$ is the {stability time and 
$\EE$ is the {\it applicability threshold}. The remarkable fact is that 
{-- $h$ being given --} the results depend only on the norm of $f$ and not on its particular form. 

Much attention has been paid in the literature in order to obtain good estimates for the quantities $\RR(\eps)$ and $\TT(\eps)$ in the different frameworks. As we shall see in the sequel, in the setting of Hölder perturbations of analytic integrable systems, the method we introduce in this paper yields sharper estimates than those that are found in the literature up to now. Before stating rigorously our results, however, it is useful to have an overview of the classical results on the effective stability of near-integrable Hamiltonian systems.
%

\paraga {\gr{The classical results.}} Let us briefly describe the classical abstract results.
 In the 70's  Nekhoroshev  proved his seminal theorem \cite{Nekhoroshev_1977}, which asserts that for a steep
real-analytic function $h$ and for any real-analytic perturbation $f$ with analytic extension to a complex domain $\cD$, 
all solutions are stable at least over exponentially long time intervals. Namely,
there exist positive {\it exponents} $a$, $b$ and a positive threshold $\EE$, depending only on $h$, 
such that if $\norm{f}_{\cD}\leq \EE$, then any initial condition
$(I_0,\theta_0)$ gives rise to a solution $\big(I(t),\theta(t)\big)$ which is defined at least for $\abs{t}\leq \exp\big(c(1/\eps)^a\big)$ and satisfies 
{$\norm{I(t)-I_0}\leq C\eps^b$}
 in that range. Here $\norm{f}_{\cD}$ is the $C^0$ sup-norm on the domain $\cD$ and 
$c$, $C$ are positive constants which also depend only on $h$. With our notation, for these systems:
\begin{equation}\label{eq:Nekhconstants0}
\TT(\eps)=\exp(c(1/\eps)^a),\qquad \RR(\eps)=C\eps^b,
\end{equation}
while the expression of the threshold $\EE$ is quite difficult to obtain explicitly\footnote{Thresholds have been studied more extensively in applications to celestial mechanics, see e.g. \cite{Niederman_1996} or \cite{Barbieri_Niederman_2020}}, see  \cite{Nekhoroshev_1977}.
Since the constants $c$ and $C$ are less significant than the exponents we will 
get rid of them in our subsequent description.

Nekhoroshev's proof is based on the construction of a partition (a ``patchwork") of the phase space into zones of 
approximate resonances of different multiplicities, over which one can construct adapted normal forms.
The global stability result necessitates a very delicate control of the size and disposition of the elements 
of the patchwork in order to produce a ``dynamical confinement'' preventing the orbits from fast motions along distances
larger than the confinement radius (see below for a discussion).

In the convex case, as noticed in \cite{Gallavotti_1986} and \cite{Benettin_Gallavotti_1986},
a shrewd use of energy conservation leads to a much simpler and ``physical'' way to confine the orbits. 
This gave rise to two distinct series of works, originating in the articles of Lochak \cite{Lochak_1993} 
- where the simultaneous approximation method was introduced - and P\" oschel \cite{Poschel_1993} - where the construction of
Nekhoroshev's patchwork was made much easier - both relying on the convexity or quasi-convexity of the integrable Hamiltonian.

The simplicity of these methods made it possible to
prove that
the Nekhoroshev Theorem in the analytic case holds with 
\begin{equation}\label{eq:Nekhconstants}
\TT(\eps)=\exp(c(1/\eps)^{1/2n}),\qquad \RR(\eps)=C\eps^{1/2n},
\end{equation}
if $h$ is assumed to be quasi-convex (see \cite{Lochak_1993, Lochak_Neishtadt_Niederman_1994, Poschel_1993}).
Moreover, besides the global result, one can state local results for neighborhoods of resonant surfaces. For
$m\in\{1,\ldots,n-1\}$, consider
a sublattice $\Lambda\in \Z^n_K:=\{k\in\Z^n:|k|_1\le K\}$ of rank $m$ and
the  resonant subset $\cM_\Lambda:=\{I\in\R^n\mid \nabla h(I)\in \Lambda^{\bot}\}$. Then, for all trajectories starting at a distance of
order~$\eps^{1/2}$ of $\cM_\Lambda$, one gets larger stability exponents,
namely $a=b=1/(2(n-m))$. Moreover, in the resonant block $\cB_\Lambda$ (which is obtained by eliminating from $\cM_\Lambda$ all the intersections with other resonant subsets $\cM_{\Lambda'}$, with $\text{rank }\Lambda'=m+1$) one can even take $a=1/(2(n-m)),\ b=1/2$.

As alluded to above, long time stability does not require {\it a priori} the analyticity of the Hamiltonian
at hand. For general Gevrey quasi-convex systems\footnote{See \cite{Marco_Sauzin_2003} for the definition.}, the fast decay of the Fourier
coefficients also yields exponentially long stability times. Namely, for $\beta$-Gevrey systems (where $\beta$ is the Gevrey exponent) it is proved in  \cite{Marco_Sauzin_2003} that
$$
\TT(\eps)=\exp\big(c/\eps^{1/(2n\beta)}\big),\qquad \RR(\eps)=C\eps^{1/(2n\beta)}.
$$
 The proof is based on a direct construction of normal
forms for Gevrey systems. This study was initiated by M. Herman for proving the optimality of the stability exponents by constructing explicit
examples taking advantage of the flexibility of the Gevrey category, see below.

Soon after, finitely differentiable systems have been investigated in \cite{Bounemoura_2010} using a {\it direct} implementation of Lochak's scheme 
in this setting, which yields the estimates
$$
\TT(\eps)=c/\eps^{(\ell-2)/(2n)}\qquad \RR(\eps)=C\eps^{1/(2n)}
$$
for quasi-convex $C^\ell$ systems with $\ell\geq2$ and integer. On the other hand, the stability of $C^\ell$ systems, with $\ell$ an integer such that $\ell\ge \ell^*n+1$ for some suitable $\ell^*\ge 1, \ell^*\in\N$, satisfying a property known as Diophantine-Morse condition\footnote{The Diophantine-Morse property is a special case of the Diophantine-steep condition introduced in \cite{Niederman_2007}\label{nota} which, in turn, is a prevalent condition on integrable systems that ensures long time stability once these are perturbed. All steep functions are Diophantine-steep.}, was investigated in \cite{Bounemoura_2011}, where the values
$$
\TT(\eps)=c/\eps^{\ell^*/[3(4(n+1))^n]}\qquad \RR(\eps)=C\eps^{1/(4(n+1))^n}
$$ 
were found. 

The case $\ell=+\infty$ has been studied in \cite{Bambusi_Langella_2020}, where the authors find that, in the case $h(I)=I^2/2$ and for fixed $b\in(0,1/2)$, for any $M>0$ there exists $C_M>0$ such that
$$
\TT(\eps)=\frac{C_M}{\eps^M}\qquad \RR(\eps)=C_M\eps^{b}\ .
$$ 
The result is achieved by implementing an innovative {\it global} normal form in Pöschel's framework. 

Finally, we also refer to the recent work \cite{Fejoz_Bounemoura_2018}
and references therein for much more information about stability in various functional classes.

\paraga {\gr{Purpose of the work.}} The objective of this paper is to make a systematic use of 
analytic smoothing methods to 
derive normal forms in a very simple way - whatever the regularity of the Hamiltonians at hand - from the usual analytic ones.
This way we  get maximal flexibility to adapt the different long-time stability proofs
to a large class of function spaces.  We will investigate here only the case of  H\"older differentiable Hamiltonians,
but our method extends to any steep functions belonging to any regularity class which admits an analytic smoothing. 
More precisely, the proposed strategy (see Section \ref{strategia parchetto}) allows us to prove, in a very simple way, the first 
Nekoroshev-type result of stability for H\"older steep Hamiltonians with presumed sharp exponents\footnote{Sharpness has the same
meaning as in \cite{Guzzo_Chierchia_Benettin_2016}, i.e. these are the best values of the exponents for ${\bf T}(\eps)$ and ${\bf R }(\eps)$ that one can obtain with these techniques.}. In this case one cannot expect to get more than polynomial stability times 
relative to the size $\eps$ of the perturbation \cite{Bounemoura_2010}. In the course of the proof we need to adjust in a rather 
unusual way the size of the various  parameters: ultraviolet cutoff and, in an essential way, the analyticity width, as a function of  the size $\eps$ of the 
perturbation.

\paraga {\gr{Main results}.} Let us fix the main definitions and assumptions. In the following, given $\nu\in\set{1,\ldots,\infty}$, we denote by $|\cdot|_\nu$ the corresponding $\ell^\nu$-norm in $\R^n$ or $\C^n$.
	{We denote by $ B_\nu(I_0,R)$ the open ball centered at $I_0$ of radius $R$ for the norm $\norm{\, \cdot\,}_\nu$ {in $\R^n$}.} 
	
Consider a Hamiltonian of the form (\ref{eq:hampert}),  {where we assume, {\it for the sake of simplicity}, that the unperturbed part $h$ is analytic\footnote{As we will see in the course of the proof, assuming that $h$ is Hölder with large enough exponents would be enough, see Section~\ref{ssec:easy}} while only the perturbation $f$ is Hölder, so:}
\begin{equation}\label{hef}
h\in C^\omega(\pallara_{\rho_0}),
\qquad f\in C^\ell(\pallara\times\T^n),
\end{equation}
where
$\pallara_{\rho_0}$ is the complex extension of analyticity width $\rho_0\ge 1$ of $\pallara$, and $\ell\in (1,+\infty)$ (meaning that $f$ is H\"older differentiable
when $\ell$ is not an integer, see section \ref{func_setting} for a brief overview on this class of functions). The small parameter is
\begin{equation}\label{eps}
\eps:=
|f|_{C^\ell(\pallara\times\T^n)},
\end{equation}
(see \eqref{norma holder} for a definition of the Hölder norm). We denote by $\om=\nabla h:\R^n\to\R^n$ the action-to-frequency map attached to $h$.

\vskip1mm

We will assume that the Hessian of $h$ is uniformly bounded from above:
\begin{equation}
\label{M}
M := \sup_{I\in \pallara_{\rho_0}} \norma{D^2 h(I)}_{op} 
< \infty,
\end{equation}
where $\norma{\ }_{op} $ stands for the operator norm induced by the 
Hermitian norm on $\C^n$. 

\vskip1mm

We will also assume that the Hamiltonian $h$ is  steep
according to the following definition. 
\begin{defn}[Steepness]\label{def steep} Fix $\delta>0$.
A $C^1$ function $h: B_\infty(0, R+\delta)\to\R$ is \co{steep} with \co{steepness indices} $\alg_1,\ldots,\alg_{n-1}\ge 1$ and \co{steepness coefficients} $C_1,\ldots, C_{n-1}, \delta$ if:
\begin{enumerate}
\item $\inf_{B_\infty(0,R)} \abs{\om(I)}_2 > 0$;
\vskip1mm
\item for any $I \in B_\infty(0,R)$ and any $m$-dimensional subspace $\Gamma$ orthogonal to $\om(I)$, with $1\le m < n$:
\begin{equation}
\max_{0\le \eta \le \xi} \,  \min_{u\in\Gamma,\abs{u}_2=\eta} \abs{\pi_\Gamma\om(I+u)}_2 > C_m \xi^{\alg_m},\quad \forall \xi \in (0,\delta],
\end{equation}
where $\pi_\Gamma$ stands for the orthogonal projection on $\Gamma$.
\end{enumerate}

\end{defn}

\begin{rmk}
Note that a uniformly strictly convex function is steep with steepness indices equal to $1$.
\end{rmk}
\begin{rmk}
	The steepness condition is generic in the space of jets of sufficiently regular functions (see \cite{Nekhoroshev_1973} for the general discussion and \cite{Schirinzi_Guzzo_2013}, \cite{Barbieri_2020} for sufficient conditions for steepness in the space of jets of order four and five respectively).
\end{rmk}

Our main theorem is the following.
%

{
	\begin{thm}[Stability estimates in the steep case]\label{MainTh3}		
Consider a near-integrable Hamiltonian system \eqref{eq:hampert} satisfying \eqref{hef} and assume $\ell\ge n+1$\ \footnote{Actually one could probably get $\ell\gtrsim n/2$ by making use of Paley-Littlewood theory.}. Suppose
that $h$ is steep in $B_\infty(0,R)$ with steepness indices $\alg:= (\alg_1,\ldots,\alg_{n-1})$ and set:
 $$
 \mathtt{a} := \frac{\ell-1}{2n\alg_1\times\cdots\times\alg_{n-2}}+\frac12\quad, \qquad \mathtt{b} :=  \frac{1}{2n\alg_1\times\cdots\times\alg_{n-1}}\ .
 $$
 Then, there exist positive constants ${\bf E}={\bf E}(n,\ell,\alg),\tC''_{\bf I}:=\tC''_{\bf I}(n,\ell,\alg)$, $\tC''_{\bf T}:=\tC''_{\bf T}(n,\ell,\alg)$ such that, for $\eps\le{\bf E}$,
 the radius and time of confinement relative to any initial condition in the set $ B_{\infty}(0, R/4)$ satisfy:
\begin{equation}\label{tempi-raggi}
{\bf R}(\eps) \leq \tC''_{{\bf I}} \eps^\mathtt{b}
\quad, \qquad 
{\bf T}(\eps) \leq \tC''_{{\bf T}}\frac{ 1}{ |\ln\eps|^{\ell-1}\,\eps^\mathtt{a}}\,.
\end{equation}
\end{thm}
}


\begin{rmk}
	\ \newline
	$\bullet $ {The presence of the logarithm in \eqref{tempi-raggi} comes from the fact that in our method we have some freedom to fix the analyticity width depending on $\eps$, in contrast with the classical analytic setting.  We send the reader to Remark \ref{pippone}, where this comment is contextualized,  the dependence of the analyticity width in $\eps$ is made precise and a qualitative justification is given. }
	\ \newline 
	$\bullet$ If we set $\alg_1,\ldots,\alg_{n-1}=1$ (i.e. the convex case) we obtain better estimates than in \cite{Bounemoura_2010}. 
	\ \newline
	$\bullet$ {Our proof relies on the geometric construction of the geography of resonances introduced in \cite{Guzzo_Chierchia_Benettin_2016}, which is appropriate only for Hamiltonians in $n \ge 3$ degrees of freedom.  Here too we shall restrict to this setting,  the $2$ d.o.f.  isoenergetic non-degenerate case being easily managed through KAM theory.  A specific construction should be implemented to treat the peculiarity of the isoenergetic degenerate $2$ d.o.f. case.  This study is in progress in a forthcoming work.  }
\end{rmk}

\paraga {\bf Prospects.} 
%
%
%
	
	The sharpness of the exponents in Theorem \ref{MainTh3} should be proved in the same way as in the case of convex system. The first attempt to tackle this problem led  to work in the Gevrey category instead of the analytic one and construct examples with unstable orbits, which experience a drift in action of the same order as the confinement radius within a time of the same order as the stability time, see \cite{Marco_Sauzin_2003}. It has then be realized that
the initial conjecture in quasi-convex analytic systems ($a\sim 1/2n$, see \cite{Chirikov_1979} and Lochak \cite{Lochak_1993}) was in fact incorrect: as proved in \cite{Bounemoura_Marco_2011} using a purely topological argument together with
the previous remark on the {\it local} exponents near simple resonances, one can choose $a=1/(2(n-1))$ as a global stability exponent for $\TT(\eps)$.  This result was improved soon after with $a\sim1/(2(n-2))$ (see \cite{Zhang_2017}). The construction of unstable system proving
the optimality of these latter exponents was achieved in \cite{Marco_Sauzin_2003}, \cite{Lochak_Marco_2005}, \cite{Zhang_2017}.
A remarkable fact is that the unstable mechanism introduced by Arnold in the 60's, with its
subsequent improvements, is exactly what is needed to produce the unstable examples in the quasi-convex case.

As for the steep case, a careful construction of the geography of resonances leads with strong evidence to the conjecture that the exponents 
$a=1/(2n\alg_1...\alg_{n-2})$ and $b=1/(2n\alg_1...\alg_{n-1})$ are sharp (see ref. \cite{Guzzo_Chierchia_Benettin_2016}).
The question of constructing explicit examples with unstable orbits proving this sharpness is still open nowadays and is maybe
the last challenging problem in the general long time stability theory, probably relying on new Arnold diffusion ideas.

\vskip1mm


The paper is organized as follows: in the next section  we give a short overview of the classical methods with particular attention on the geometry of resonant blocks,  on which the present work strongly relies.  Next we define the functional setting.  In Section \ref{analytic_smoothing} we introduce the analytic smoothing appropriately adapted to our problem.  
Finally Section \ref{steep case} is devoted to the study of the steep case.

\subsection*{Acknowledgements}
We wish to thank A. Bounemoura, L. Biasco, L. Chierchia, M. Salvatori, and L. Niederman for fruitful discussions and stimulating comments, which definitively helped to improve this work. 
J.E.M.
 acknowledges the support of the INdAM-GNAMPA 
grant ``{\it Spectral and dynamical properties of Hamiltonian systems}''.

\section{General setting and classical methods: a geometric framework}\label{pippo}


\paraga {\bf Resonances, resonant normal forms and the steepness condition.}\label{aiuto} 
Consider a Hamiltonian system of the form \eqref{eq:hampert} defined on $O\times \T^n$, where $O$ is an open subset of $\R^n$.
The main feature underlying Hamiltonian perturbation theory is that one can modify the form of the perturbation $f$ by composing $H$ with 
properly chosen local Hamiltonian diffeomorphisms, in order to remove a large number of ``nonessential harmonics''. 
The result of this process - a local normal form - strongly depends on the location of the domain of the normalizing diffeomorphism w.r.t the
resonances of the unperturbed part $h$, and enables one to discriminate between fast drift and extremely slow drift directions in the action space, 
according to this location.

\vskip1mm

Let us first make this idea more precise. Given an integer lattice $\Lambda\subset \R^n$ of dimension $m\in\{1,\ldots,n-1\}$ -- a {\it resonance lattice} --
one associates with $\Lambda$ the resonance vector subspace $\Lambda^\bot\subset\R^n$ in the frequency space $\R^n$,
together with the corresponding resonance subset in the action space previously introduced
$$
\cM_\La:=\om^{-1}(\Lambda^\bot)=\{I\in O\mid \om(I)\in \Lambda^{\bot}\},
$$ 
where $\om=\nabla h$ is the frequency map. The dimension $m$ of $\La$ is said to be the multiplicity of the resonance $\cM_\La$.  Of course, 
given a resonance module $\La'\supset\La$ with $\dim\La' >\dim\La$, the resonance $\cM_{\La'}$ is contained in $\cM_\La$,
so that a resonance subset contains in general infinitely many resonances of higher multiplicity.
The complement $\cM_0\subset O$ of the union
of all resonance subsets is the {\it non-resonant subset}. 
In general, a resonance subset $\cM_\La$ has no particular structure, however,
one can think of $\cM_\La$ as a submanifold of $\R^n$ of the same dimension as $\Lambda^\bot$ (with perharps singular loci).

\vskip1mm

As a rule, when $\eps$ is small enough,
for a small enough $\eps$-depending
neighborhood $W_\Lambda$ of the parts of the resonance subset $\cM_\La$ {\it located far enough from resonances of higher multiplicity}\footnote{In fact, only
a finite $\eps$-depending subset (related to the cutoff $K(\eps)$ introduced below) of these resonances has to be taken into account.}, 
one can iteratively  construct a symplectic diffeomorphism $\Psi_\Lambda$, whose image contains $W_\Lambda\times\T^n$,
such that the pull-back $H_\Lambda=H\circ\Psi_\Lambda$ takes the following form
\beq\label{eq:exnormform}
H_\Lambda=h+N_\La+R_\La.
\eeq 
Here $R_\La$ is a remainder whose $C^2$ norm is (very) small\footnote{The smallness depends on the regularity of the system.} with respect to $\eps$ and the
{\it resonant part} $N_\Lambda$ contains only harmonics belonging to $\La$, that is:
$$
N_\La(I,\th)=\sum_{k\in\La,\,\abs{k}_1\leq K(\eps)}a_k(I)\,e^{i k\cdot\th},
$$
where $K(\eps)$ is an ultraviolet cutoff which has to be properly chosen\footnote{This choice is indeed a main issue in the theory.}. 
Both terms $N_\La$ and $R_\La$ of course depend on $\eps$.
A subset $W_\La$ for which such a normal form is proved to exist will be called a { {\it normal form neighborhood associated with $\La$},
with multiplicity $\dim\La$.  One proves that   the space of actions can be covered by such neighborhoods, and in
Section \ref{geometria}, we will construct finer covers by subsets of those, named \co{resonant blocks} (and denoted by $D_\Lambda$ in the aforementioned section). }

\vskip1mm 

The iterative process to construct the normalizing diffeomorphism involves the control of small denominators which appear 
during the resolution of the so-called homological equation, and which depend on the location of the normalization domain with respect
to the resonances (see for instance \cite{Poschel_1993}). This can be seen as a drawback of the method which
could be greatly simplified by an idea due to Lochak (see below), however the general method presented here give precise dynamical 
informations which would not be reachable otherwise.

\vskip1mm

The Hamilton equations generated by \eqref{eq:exnormform} yield the following form for the evolution of the action variables:
\beq\label{eq:evolactions}
\begin{array}{lll}
I(t)-I(0)&=&\dst\int_0^t \d_\th N_\La\big(I(s),\th(s)\big)+\d_\th R_\La\big(I(s),\th(s)\big)\,ds\\[8pt]
&=&\dst \sum_{k\in\La,\,\abs{k}_1\leq K(\eps)}k\cdot \Big(\int_0^t i\,a_k(I(s))\,e^{i k\cdot\th(s)}\,ds\Big)+ \cR(t).
\end{array}
\eeq
The variation of  $I$ is therefore the sum of the main part 
\beq\label{eq:drift2}
\cD(t):=\sum_{k\in\La,\,\abs{k}_1\leq K(\eps)}k\cdot \cN^{(k)}(t),   \qquad  \cN^{(k)}(t)=\int_0^t i\,a_k(I(s))\,e^{i k\cdot\th(s)}\,ds,
\eeq
and the very small remainder term $\cR(t)$.

\vskip1mm

To simplify the presentation in the following, we will {\it forget about the angles} and consider only the action part of the solutions of our system (which
is legitimized by the fact that the angles play no role in the various estimates).

\vskip1mm

The whole theory relies firstly on the obvious fact that the main drift term $\cD(t)$ in~(\refeq{eq:drift2}) belongs to the vector space $\vect \La$ spanned by $\La$
(which is often called {\it ``plane'' of fast drift}), and secondly on the smallness of the remainder term~$\cR$.
A solution starting from some initial condition $I(0)\in W_\La$ will therefore remain very close to the fast drift
space
$$
I(0)+\vect \La
$$
during a very long time -- governed by the smallness of $\cR$ -- 
{\it as long as it is contained inside the neighborhood~$W_\La$}.
This makes it necessary to understand first the intersections of the fast drift planes $I+\vect\La$ and the neighborhoods $W_\La$
to which they are attached.

\vskip1mm

As an extreme example, let us consider the Hamiltonian
$$
h(I)=\demi (I_1^2-I_2^2)
$$
on $\A^2$, with (invertible) frequency map $\om(I_1,I_2)=(I_1,-I_2)$. 
We focus on the resonance module $\La=\Z(1,-1)$, so that $\La^\bot=\R(1,1)$ and $\vect \La=\cM_\La$. Hence,
given an initial action $I(0)\in \cM_\La$, the entire fast drift affine subspace $I(0)+\vect \La$ coincides with $\cM_\La$, so that
nothing prevents the fast drift to take place during the whole motion provided the perturbation is well-chosen: the resonance $\cM_\La$ is called a {\it superconductivity channel}. 
No long time stability result can be expected
in this case: indeed, when $f(I,\th)=\sin (\th_1-\th_2)$,
 the initial condition $I=0$, $\th=0$ yields the fast evolution $(I_1(t),I_2(t))=(-\eps t,\eps t)$ for the action variables \footnote{Here a proper choice of the initial angles is needed.}.

\vskip1mm

In constrast with the previous example, for the Hamiltonian 
$$
H(I,\th)=\demi {\norm{I}}^2_2+\eps f(\th)
$$
on $\A^n$, for any $\La\subset\Z^n_K$, the the resonant set $\cM_\La$ coincides with $\La^\bot$,
so that the affine planes of fast drift are always orthogonal to $\cM_\La$. In this case a fast drift - if it happens - makes the orbits move away
from the resonance in a very short time.

\vskip1mm

These extreme examples illustrate the role of the Nekhoroshev condition: steepness is an intermediate quantitative property,
which prevents from the existence of the superconductivity channels by ensuring a certain amount of transversality
between the fast drift planes and the corresponding resonances in action. Starting from an action $I=I(0)$ located at some
resonance $\cM_\La$, so that its associated frequency $\om(I)$ is orthogonal to $\Gamma:=\vect\La$,  the condition
\beq\label{Neko1}
\max_{0\le \eta \le \xi} \,  \min_{u\in\Gamma,\abs{u}_2=\eta} \abs{\pi_\Gamma\om(I+u)}_2 > C_m \xi^{\alg_m},\quad \forall \xi \in (0,\delta],
\eeq
(where $\pi_\Gamma$ stands for the orthogonal projection on $\Gamma$) imposes that a drift of length $\xi$ starting from $I$ and
occuring along the fast drift plane $I+\Gamma$ makes the projection $\pi_\Gamma(\om)$ change by an amount of   $C_m\xi^{\alg_m}$ 
during the way. 

\begin{figure}[h]
\centering
\vskip-2mm
\includegraphics[width=.8\textwidth]{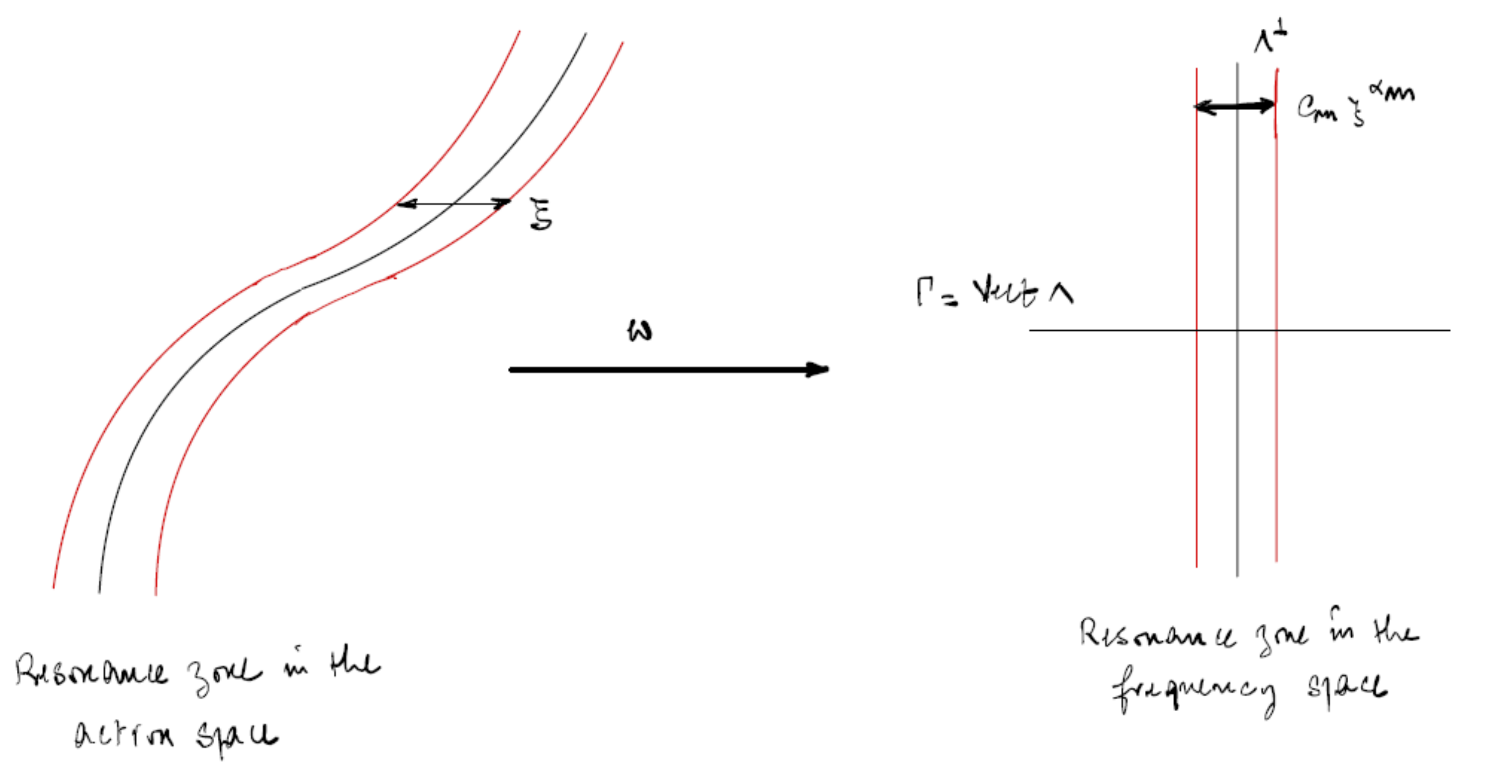}
\vskip-10mm
\caption{Geometric interpretation of the steep condition}\label{Figure 1}

\end{figure}

This admits an easy geometric interpretation (see Figure 1). Assume $\dim\La=m$ and consider the vector space $\Gamma$ 
spanned by $\La$,
together with its orthogonal space $\La^\bot$ - of dimension $n-m$. Then one can define a family of tubular neighborhoods of $\La^\bot$
of width $\delta>0$ by
\beq\label{eq:tubular}
\bT_\delta(\La^\bot)=\{\varpi\in\R^n\mid \pi_\Gamma(\varpi)<\delta\},\quad \delta>0\,.
\eeq
Each such neighborhood gives rise to a neighborhood of the resonance $\cM_\La$ in action, namely:
\beq\label{eq:neighb}
\bW_\delta(\cM_\La)=\om^{-1}\big(\bT_\delta(\La^\bot)\big).
\eeq
Therefore, condition \eqref{Neko1} just says that {\it any} orbit starting from $I$ and drifting to a distance $\xi$ from $I$
along the plane of fast drift $\Gamma$ must exit the neighborhood
$
\bW_{\delta}(\cM_\La)$
with
$ \delta=C_m\xi^{\alg_m}.
$

Note finally that given disjoint subsets $\bT$, $\bT'$ of tubular neighborhoods of the form \eqref{eq:tubular}, the associated neighborhoods
$\om^{-1}(\bT)$ and $\om^{-1}(\bT')$ are disjoint too, whatever the geometric assumptions on the frequency map~$\om$.

\paraga {\bf Nekhoroshev's hierarchy.} 
This section is inspired by Nekhoroshev's ideas as presented in the very nice paper \cite{Guzzo_Chierchia_Benettin_2016}. We also refer to \cite{Guzzo_2015} for further details and to \cite{Niederman_2007} for a  different approach. 
Nekhoroshev's strategy to prove long-time stability results for perturbations of steep Hamiltonians is
based on the previous description of resonant neighborhoods, and relies on  the following key observation. 

\vskip2mm

{\it Given $\eps$ small enough,  there exist $T(\eps)$, $R(\eps)$ and a covering of the action space $O$ by resonant ``blocks'' 
$(\cB_{m,p})_{0\leq p\leq p_m}$, for ${0\leq m\leq n-1}$,  and $m, p, p_m\in\N$, which satisfy
the following properties:
\begin{enumerate}
\item $T(\eps)\to+\infty$ and $R(\eps)\to 0$ when $\eps\to 0$;
\item each block $\cB_{m,p}$ is contained in a resonant neighborhood of multiplicity $m$ and admits an 
enlargement $\ha \cB_{m,p}\supset\cB_{m,p}$ contained in the same resonant neighborhood;
\item any solution starting from an initial condition in $\cB_{m,p}$ either stays inside $\ha \cB_{m,p}$ for  $0\leq t\leq T(\eps)$ or admits
a first exit time $t_1$ such that $I(t_1)$ belongs to a block $\cB_{m',p'}$ with $m'<m$;
\item for any initial condition $I(0)$ inside a block $\cB_{m,p}$ and for any interval $\cI$ such that $I(t)\in\ha \cB_{m,p}$ for all $t\in \cI$, then
$$
\norm{I(t)-I(0)}_2<R(\eps),\quad \forall t\in \cI.
$$
\end{enumerate}
}

\vskip3mm

We say that $m$ is the multiplicity of the block $\cB_{m,p}$. 
Taking the previous observation for granted, the stability of the action variable over a timescale $T(\eps)$ is easy to prove by finite induction.
Given an initial condition $I(0)$ located in some block $\cB_{m_0,p_0}$, either $I(t)\in \ha \cB_{m_0,p_0}$ for $0\leq t\leq T(\eps)$, or there is a $t_1$
such that $I(t)\in \ha \cB_{m_0,p_0}$ for $0\leq t <t_1$ and $I(t_1)$ belongs to a block $\cB_{m_1,p_1}$ with $m_1<m_0$. Consequently,
there is a finite sequence $(m_0,p_0),\ldots,(m_j,p_j)$ such that $m_0>m_1>\cdots>m_j$ (with maybe $m_j=0$) and a finite sequence of times
$t_0=0<t_1<\cdots<t_p=T(\eps)$ such that for $0\leq i <j$:
$$
I(t)\in\ha \cB_{(m_i,p_i)},\quad \forall t\in[t_i,t_{i+1}].
$$
In words, any orbits crosses a finite number of enlarged blocks during the interval $[0,T(\eps)]$ and get trapped inside the last one. 
To conclude, one just has to use property (4), which proves that the distance between $I(0)$ and $I(t)$ is at most $nR(\eps)$ for $t\in[0,T(\eps)]$.

\vskip1mm

One should be aware that the covering by the blocks is {\it not} a partition of $O$: two distinct blocks may have a nonempty intersection.
However, one can {\it choose} the blocks visited by the orbits according to a hierarchical order, in such a way that their multiplicity decreases as $t$ increases
\footnote{This raises the question of the existence of local  finite time Lyapunov functions on the phase space, a still unclear issue.}.
We say that a covering of $O$ by blocks satisfying the previous properties is a Nekhoroshev patchwork.

\paraga {\bf Construction of Nekhoroshev patchworks.}  
Let us now describe how the blocks are constructed so as to possess their covering and confinement properties\footnote{A source of inspiration
for nowadays governments.}.

\vskip1mm

Given $\eps>0$, we first fix an ultraviolet cutoff $K(\eps)$ and consider only the set $\bM_\eps$ of resonance modules  which are spanned by vectors
of length smaller than $K(\eps)$.
Given a resonant module $\La\in\bM_\eps$ of multiplicity $m$, we start with the resonant zone of ``width'' $\delta_\La$
$$
Z_\Lambda := W_{\delta_\La}(\cM_\La)=\om^{-1}\big\{\varpi\in\R^n\mid \norm{\pi_\Gamma(\varpi)}_2<\delta_\La\big\},
$$
where $\delta_\La$ has to be properly chosen as a function of $\eps$ and the various geometric invariants of the module (see section \ref{steep case}). 
We then define the ($\eps$-dependent) resonant zone $\cZ_m$ of multiplicity $m$ as
$$
Z_m=\bigcup_{\La\in\bM_\eps,\,\dim \La=m}Z_\Lambda.
$$
Given $\La\in\bM_\eps$, $\dim\La=m$, the block attached to $\La$ is obtained by removing from $Z_\La$ its intersection with the complete 
resonant zone of multiplicity $m+1$:
$$
\cB_\La= Z_\La\setminus Z_{m+1}.
$$

\vskip1mm

The blocks $\cB_{m,p}$ are the connected components of $Z_m$. With no great loss of generality, one can think of (the closure of) a block as a submanifold with 
boundary and corners -- even if it is not necessary. 

\vskip1mm

The following figure shows the construction of the blocks in the case $n=3$ (and in a transverse section). 
The resonance zone of multiplicity~2 if the disjoint union of the blue blocks, the resonance zone of multiplicity 1 is the union on the strips with red boundaries, 
while the $0$-multiplicity zone is the complement of the $1$-multiplicity zone.

\vskip1mm

In any case, the blocks satisfy two main properties.

\begin{itemize}
\item The closures of two different blocks can intersect only when their multiplicities are distinct.
\end{itemize}

This comes from a very careful choice of the widths of the various resonance zones (see \cite{Guzzo_Chierchia_Benettin_2016} and Section \ref{steep case}), which in fact 
ensures a more stringent (and crucial) property: the enlargement of a block contained in some $\cB_\La$ cannot intersect {\it any other block} contained in the zone $\cB_\La$,
neither {\it any other neighborhood $\cM_{\La'}$ with $\dim\La'=\dim\La$} (see below for precisions on the construction of the enlargement).

\begin{figure}[h]
\centering
\vskip-2mm
\includegraphics[width=.8\textwidth]{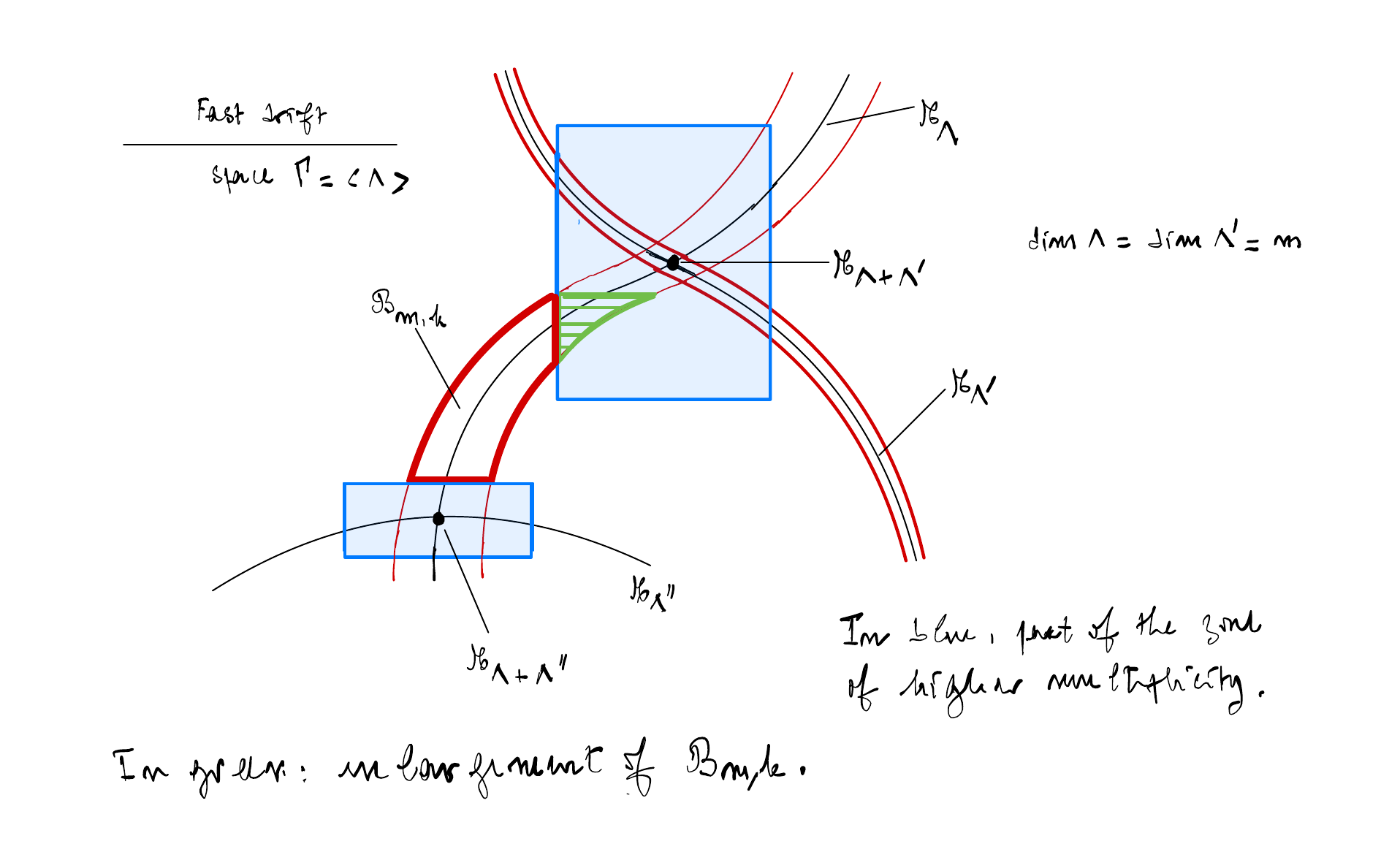}
\vskip-5mm
\caption{Construction of the resonant blocks}
\end{figure}

We state the second property in the spirit of Conley's isolating blocks theory.

\vskip1mm

\begin{itemize}
\item The frontier $\partial \cB_{m,p}$ of $\cB_{m,p}$ is the union of two subsets 
$$
\partial \cB_{m,p}=\partial^+\cB_{m,p}+\partial^-\cB_{m,p}
$$
where $\partial^+\cB_{m,p}$ (resp. $\partial^-\cB_{m,p}$) is contained in blocks $\cB_{m',p'}$ with $m'>m$ (resp. $m'<m$).
\end{itemize}

This raises new questions which could be the starting point of a better understanding of the relations between diffusion along invariant subsets 
and long-time stability theory. Indeed, given a block $\cB_{m,p}$, a description of the (generic) features of the Hamiltonian vector field $X_{H_\eps}$ at
the frontier $\partial \cB_{m,p}$ has never been done. In particular, nothing is known on the locus where $X_{H_\eps}$ ``enters the block''
and the locus where $X_{H_\eps}$ ``exits the block''. These two subsets are crucial for the understanding of the homology of the invariant sets contained
into the blocks, following Conley's theory, and could provide one with a new tool for constructing diffusing orbits in the steep setting.

\begin{figure}[h]
\centering
\vskip-2mm
\includegraphics[width=.7\textwidth]{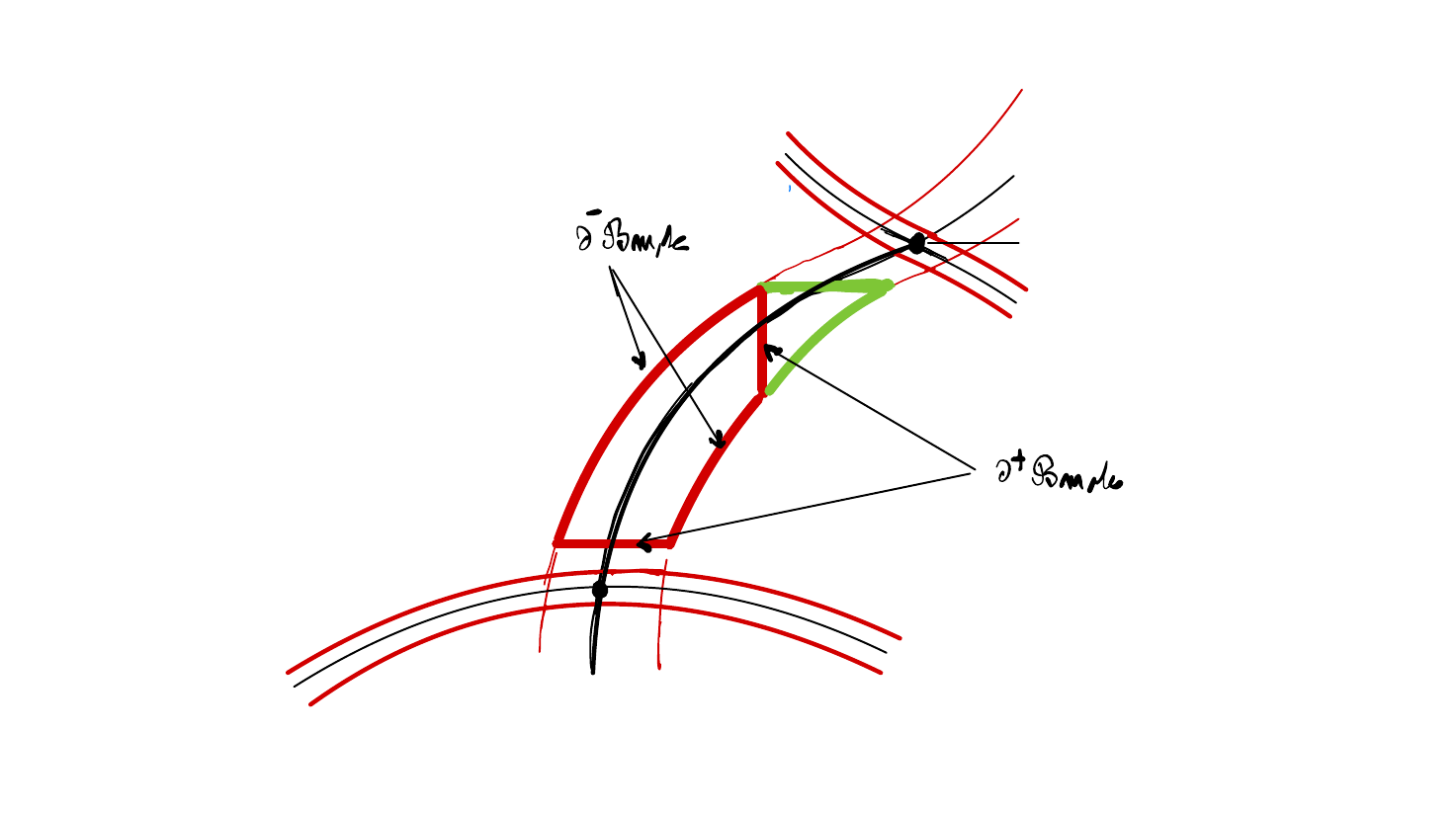}
\vskip-10mm
\caption{Interpretation of the resonant blocks in the light of Conley’s theory}
\end{figure}
\begin{figure}[h]
	\centering
	\vskip-2mm
	\includegraphics[width=.8\textwidth]{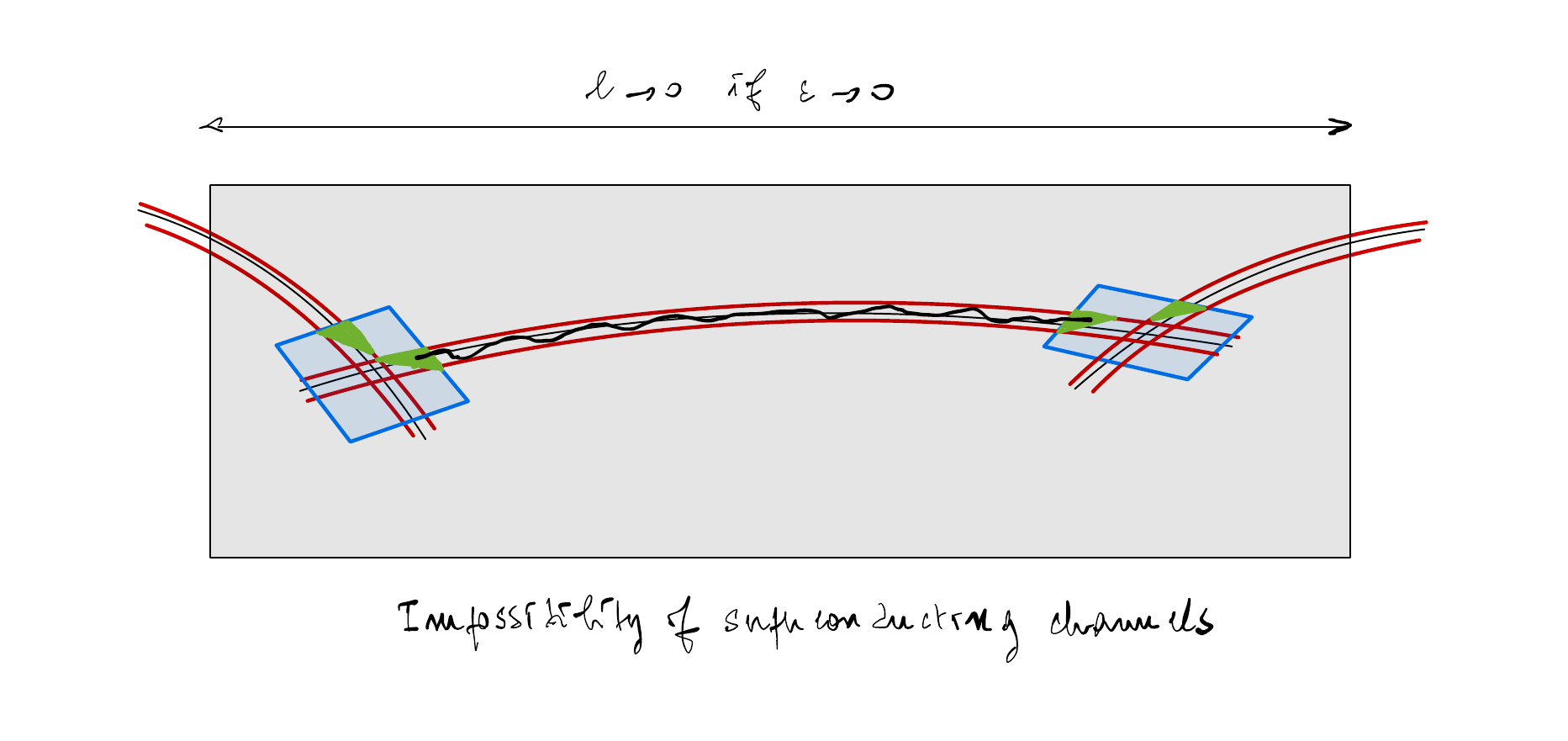}
	\vskip-5mm
	\caption{The Steepness property prevents the existence of superconductivity channels by ensuring a contact of
finite order between the resonant manifold and the plane of fast drift. Here in the figure, $\ell$ is the size of the resonant zone (see Section \ref{geometria})}
\end{figure}
Going back to the construction of Nekhoroshev's patchwork, we have to make precise the process conducting to the enlargement of a block and its
stability property. Here we will again make a crucial use of the fact that an orbit starting from an initial condition $I:=I(0)$ located in $\cB_{m,p}$ will remain extremely close to the
fast drift space $I+\vect \La$ for $0\leq t\leq T(\eps)$, as long as it stays inside the resonant neighborhood $\cM_\La$ and far enough to the higher 
multiplicity resonance zones. Hence, to enlarge the block $\cB_{m,k}$, we just have to add to it the collection of all the parts of the disks centered at points $I\in \cB_{m,p}$ which are contained in the intersection of the fast drift spaces $I+\vect\La$ with the resonant neighborhood $\cM_\La$ (the resulting added subset is
the green part in the previous two figures). 
We have in fact to add a very small neighborhood of these union of disks, in order to prevent the solutions
to exit the extended block under the influence of the remainder part $\cR$ of the dynamics during the time $T(\eps)$, but this would not change our description
significantly. Finally, one has to make sure that the extension will not intersect any other block of the same neighborhood $\cB_\La$ or any other resonance
neighborhood, which can be done by a careful tuning of the width of the zone (see Section \ref{steep case}).

This concludes our description of Nekhoroshev's method. 

\vskip1cm
\section{Functional setting}\label{func_setting}
{For $n\geq 1$, we denote the standard $n$-dimensional torus by $\T^n = \R^n/2\pi\Z^n $ and the standard $2n$-dimensional annulus by
	$ \A^n = \R^n\times \T^n$. 

	\paraga {\bf H\" older differentiable functions.} Given an integer $q\geq0$ and an open subset $D$ of $\R^n$,
	we denote by $C^{q}(D)$ the set of $q$-times continuously differentiable maps 
	$f:D\to \R$ ($C^0(D)$ being the set of continuous functions on $D$). 
	We identify $C^{q}(\T^n)$ with the subset of $C^{q}(\R^n)$ formed by the functions that are $2\pi\Z^n$-periodic
	and $C^{q}(D\times \T^n)$ with the subset of $C^{q}(D\times \R^n)$ formed by the functions 
	which are $2\pi\Z^n$-periodic with respect to their last $n$ variables.
	
	We use the conventional notation for partial derivatives: given $f\in C^{q}(D)$ and $\alp\in \N^n$, we set for $x\in D$:
	$$
	\partial^\alp f(x)=\frac{\partial^{\norms{\alp} }f}{\partial x_1^{\alp_1}\ldots \dr x_n^{\alp_n}}(x),
	$$
	with $\norms{\alp}=\alp_1+\cdots+\alp_n$.

	We denote by  $C_b^q(D)$ the set of $f\in C^q(D)$ such that
	\begin{equation}\label{norma texas}
	\norma{f}_{C^q(D)} := \sup_{\abs{\alpha} \le q}\, \sup_{x\in D} \abs{\partial^\alpha f(x)}<+\infty,
	\end{equation}
	so that $\big(C_b^q(D),\norma{\cdot  }_{C^q(D)}\big)$ is a Banach space with multiplicative norm\footnote{That is, satisfying an inequality
		of the form $\norm{fg}\leq C\norm{f}\norm{g}$ for a suitable constant $C$.}. It is understood that, for a function defined on a compex domain $D$,  the $\norma{\cdot}_{C^0(D)}$ is the usual sup-norm.  
	

	If $\ell>0$ is a non-integer real number, we write $q:=\floor{\ell}$ for its integer part and $\mu = \ell - q \in (0,1)$
	for its fractional part. 
	Given a non-negative integer $q$ and $\mu\in(0,1)$, 
	we denote by $C_b^{q,\mu}(D)$  the space formed by those functions
	$f\in C^q(D)$ such that 
	\begin{align}\label{norma holder}
	\begin{split}
	|f|_{C^{q,\mu}(D)}:=\norma{f}_{C^q(D)}+\sup_{\alpha\in\mathbb{N}^n:|\alpha|= q}\ \ \sup_{\substack{x,y \in D:\\ 0<|x-y|<1}}\frac{|\partial^\alpha f(x)-\partial^\alpha f(y)|}{|x-y|^{\mu}}<+\infty.
	\end{split}
	\end{align}
	{It is well-known that  $\big(C_b^{q,\mu}(D),\norms{\cdot }_{C^{q,\mu}(D)}\big)$ is also a Banach space with multiplicative norm. Functions belonging to these spaces are called Hölder-differentiable functions.}
	
	Given a non-integer real number $\ell>0$, together with its integer part  $q:=\floor{\ell}$ and its fractional part $\mu = \ell - q \in (0,1)$,
	we also write $C^\ell_b(D)$ instead of $C_b^{q,\mu}(D)$ and $|\cdot |_{C^{\ell}(D)}$ instead of $|\cdot |_{C^{q,\mu}(D)}$. 
	Clearly $C^\ell_b(D)\subset C^{\ell'}_b(D)$ when $\ell\geq \ell'$ and if $f\in C^\ell_b(D)$
	\begin{equation}\label{norma ineg}
	|f|_{C^{\ell'}(D)}\leq |f|_{C^{\ell}(D)}.
	\end{equation}
	

	

	\paraga {\bf Domains and their complex extensions.}

	
	Let us define the complex $n$-dimensional torus $\T_\C^n$ and the complex $2n$-dimensional annulus $\A_\C^n$ as
	\begin{equation}
	\label{infiniti}
	\T^n_\C = \C^n/2\pi\Z^n \qquad \text{and} \qquad \A^n_\C = \C^n\times \T^n_\C.
	\end{equation}
	We use angle coordinates $\theta$ on $\T^n_\C$ (with the usual abuse $\theta\in\C^n$ when there is no ambiguity)
	and action-angle coordinates $(I,\theta)$ on $\A_\C^n$. We see $\T^n_\C$  
	as a real $n$-dimensional vector bundle over $\T^n$.
	Consequently, we write
	\begin{equation}\label{norme complex}
	\abs{\teta} := \max_j \pa{\abs{\Im \teta_j}}\ , \qquad \abs{I}:= \max_j \abs{I_j}\ , \qquad \abs{(I, \teta)} = \max \pa{\abs{I},\abs{\teta}}\, .
	\end{equation}
	For  integer vectors $k\in\Z^n$, we use the ``dual" $\ell^1$-norm, which we write $|k|$ only when there is no risk of confusion.
	
	\vskip2mm
	
	We need to introduce specific domains in $\A^n_\C$. First, given $r>0$, for 
	a domain $D\subset\mathbb{R}^n$, we set
	\begin{equation}\label{dominos}
	D_r:=\big\{z\in\mathbb{C}^n:\exists z^*\in D: |z-z^*|_2< r \big\}\ .
	\end{equation}
	As for the torus, given $s>0$, we introduce the global complex neighborhood
	\begin{equation}\label{toros}
	\T^n_s := \big\{\theta\in\T^n_\C\, : \, |\teta| < s\big\}.
	\end{equation}
	We will essentially deal with complex domains of the form
	\begin{equation}\label{cilindros}
	\ciccio{r,s}\ := D_r\times\T^n_s\subset \A_\C^n.
	\end{equation}
	We finally write $D^\R_r$ and $\ciccio{r,s}^{\R}$ for the projections of $D_r$ and $\ciccio{r,s}$ on $\R^n$ and $\A^n$ respectively.\\
	
	

	\paraga {\bf Analytic functions and norms.} If $g$ is a bounded holomorphic function defined on $\T^n_s, D_r$ or $\ciccio{r,s}$ we  
	denote the corresponding classical sup-norms by 
	\begin{equation}\label{norme analitiche}
	\abs{g}_s = \sup_{\teta\in\T^n_s} \abs{g(\teta)},\quad \abs{g}_{r} = \sup_{I\in D_r} \abs{g(I)},\quad \abs{g}_{r,s} = \sup_{(I,\teta)\in\ciccio{r,s}} \abs{g(I,\teta)}.
	\end{equation}
	Fix a bounded holomorphic function $g:\cD_{r,s+2\sigma}\to \C$, where $\sigma >0$,  and let $g(I,\teta) = \sum_{k\in\Z^n} \hat{g}_k(I)e^{\im\,k\cdot\teta}$ be its Fourier expansion, where $k\cdot\teta = k_1\teta_1 + \cdots + k_n\teta_n$.  We then introduce the {\it weighted Fourier norm}
	\begin{equation}\label{fourier_norm}
	\normf{g}{r}{s}:= \sup_{I \in D_r}\sum_{k\in\Z^n}\abs{\hat{g}_k(I)}\,e^{\abs{k}s},
	\end{equation}
	which is finite and satisfies
	\begin{equation}\label{norme Fouriertiche}
	\abs{g}_{r,s} \leq \normf{g}{r}{s}\leq {\rm coth^n\,}\sigma\,\abs{g}_{r,s+\sigma}.
	\end{equation}

	We denote by $\cA_{r,s}$ the space of holomorphic functions on $\ciccio{r,s}$ with finite Fourier norm.
	Endowed with this norm, $\cA_{r,s}$ is a Banach algebra. 
	
	Finally, the norm of a vector valued function will be the maximum of the norms of its components.

	
	\section{Analytic smoothing}\label{analytic_smoothing}
	We state in this section the key ingredient of the present work. We first recall the analytic smoothing method as 
	developed by Jackson-Moser-Zehnder for H\"older functions of $\R^n$: 
	given a H\"older function $f\in C^\ell(\R^n)$ and a positive number $s\le 1$, this yields an analytic 
	function on the complex neighborhood 
	$\R^n_s$ whose restriction to $\R^n$ is close to $f$ in the $C^k$ topology, for $1\leq k\leq \ell$. 
	
	We then adapt their method to our specific setting of functions defined on $\A^n$ (see Section \ref{smooth anello})  and,  in addition,  we derive the new estimate \eqref{fourier_est_lissage} for the weighted Fourier norm of the smoothed function.


	\subsection{Analytic smoothing in $\R^n$} 
	We recall here the result by
	Jackson, Moser and Zehnder, following the presentation by \cite{Chierchia_2003} and \cite{Salamon_KAM}.

	\begin{prop}[Jackson-Moser-Zehnder] 
		\label{liscio chierchia} Fix an integer $n\geq1$, a real number
		$\ell> 0$ and let $f\in C^\ell_b(\R^n)$. Then there is a constant $\jack = \jack(\ell, n)$ such that for every $0<s\le 1$ 
		there exists a function $\tf_s$, analytic  on $\R^n_s$, which satisfies 
		\begin{equation}
		\left|\partial^\alpha \tf_s(x) - \sum_{\substack{\beta\in\mathbb{N}^n:|\beta|\leq \floor{\ell}-|\alpha|}}\partial^{\alpha+\beta} f(\Re x)\frac{(\Im x)^\beta}{\beta!}\right|\le \jack\, s^{\ell-|\alpha|} \abs{f}_{C^\ell(\R^n)},\quad \forall x \in \R^n_s,
		\end{equation} 
		for all multi-integer $\alpha\in\N^{n}$ such that $\abs{\alpha}\le \floor{\ell}$.  More precisely, given any even $C^\infty$ function $\Phi$ with compact support in 
		$\R^n$
		and setting \begin{equation}\label{Kappa}
		K(\xi) := \frac{1}{(2\pi)^{n}}\int_{\R^{n}} \Phi(x) e^{\im x\cdot \xi}dx ,\quad \xi\in \R^n_s,
		\end{equation}
		\newline
		the function
		\begin{equation}\label{liscio}
		\tf_s( x):=\int_{\R^{n}}{K}\left(\frac{x}{s}-\xi\right)f(s\xi)\ d\xi \ ,
		\end{equation}
		satisfies the previous requirements (where the constant $\jack(\ell, n)$ depends on the choice of $\Phi$).
	\end{prop}
	
	Observe that $\tf_s$ takes real values when its argument is in $\R^n$.

	
	\subsection{Analytic smoothing in $\A^n$.}\label{smooth anello}
	In the following,  the H\"older regularity $\ell$ is assumed to satisfy $\floor{\ell} \ge n+1$ as in the hypotheses of Theorem \ref{MainTh3}. \\ We now specialize the previous result to our setting and give a more detailed description of the method in the
	case of functions of $\A^n$. In that case, the analytic smoothing is a truncation of the Fourier series of the initial Hölder function with suitably modified Fourier coefficients (the so-called Jackson polynomials).  Our main concern here is to derive an estimate on the weighted Fourier
	norm of an $s$-smoothed $C^\ell$ function over a complex strip of width $s$.
	
	To make the whole presentation more explicit {\it and take the anisotropy of the weighted Fourier norm into account}, we first consider functions defined on $\R^n$ and $\T^n$ separately. This then yields a statement for functions of $\A^n$.
	
	\vskip2mm
	
	$\bullet$ {\it The non-periodic case}. Fix an even function $\Phi:\R^n\to[0,1]$, of class $C^\infty$, with support in the ball $\ov B_2(0,1)$ and 
	let $K:\C^n\to\C$ be its Fourier-Laplace transform:
	\beq\label{eq:kernel}
	K(y)=\frac{1}{(2\pi)^n}\int_{\R^n}\Phi(\eta)e^{-i \eta\cdot y}d\eta.
	\eeq
	Since $\Phi$ is compactly supported, then $K$ is an entire function . Moreover its restriction to $\R^n$ is in the Schwartz class
	$\jS(\R^n)$ since $\Phi$ is, and this is also the case for the translates $y\mapsto K(y-z)$ for $y\in\R^n$ and fixed $z\in\C^n$.
	
	\vskip2mm
	
	Let $f:\R^n\to\R$ be a $C^\ell$ function with $\floor{\ell}\geq n+1$, with compact support contained in the ball $\ov B_\infty(0,R_0)$ for some $R_0>0$.
	Given $s\in\,]0,1]$, set for $x\in\R^n$:
	\beq\label{eq:smoothing}
	\f_s(x)=\frac{1}{s^n}\int_{\R^n}K\Big(\frac{x-y}{s}\Big)f(y)dy=\int_{\R^n}K\Big(\frac{x}{s}-y\Big)f(sy)dy=\int_{\R^n}K(y)f(x-sy)dy.
	\eeq
	By Fourier reciprocity:
	$$
	\f_s(x)=\int_{\R^n}\Phi(\eta)\ha{f(x-sy)}(\eta)d\eta,
	$$
	with:
	$$
	\ha{f(x-sy)}(\eta)=\frac{1}{(2\pi)^n}\int_{\R^n}f(x-sy)e^{-iy\cdot\eta}dy=\frac{1}{(2\pi)^ns^n}\int_{\R^n}f(u)e^{-i(x-u)\cdot\eta/s}du
	=\frac{e^{-i x\cdot\eta/s}}{s^n}\ha f\Big(\frac{-\eta}{s}\Big).
	$$
	Therefore, since $\Phi$ is even:
	\beq\label{fsazioni}
	\f_s(x)=\frac{1}{s^n}\int_{\R^n}\Phi(\eta)\ha f\Big(\frac{-\eta}{s}\Big)e^{-ix\cdot\eta/s}d\eta=
	\int_{\R^n}\Phi(s\eta)\ha f({-\eta})e^{-i x \cdot\eta}d\eta
	=\int_{\R^n}\Phi(s\eta)\ha f({\eta})e^{i x\cdot\eta}d\eta.
	\eeq
	Hence $\f_s$ is the inverse Fourier-Laplace transform of the ``truncation'' 
	$$
	\eta\mapsto\Phi(s\eta)\ha f(\eta).
	$$
	The first term of (\ref{eq:smoothing}) shows that $\f_s$ extends to $\C^n$ and is an entire function. 
	To get our final estimate we  go back to the second term in \eqref{eq:smoothing}, which yields
	\beq\label{boh}
	|\f_s(z)|\le \norma{f}_{C^0(\R^n)} \int_{\R^n}\left|K\left(\frac{z}{s}-y\right)\right|dy,\qquad z\in\C^n.
	\eeq
	By the Schwartz estimate of Lemma \ref{gigi},  there exists a constant $C_n$ such that 
	$$
	\left|K\left(\frac{z}{s}-y\right)\right|\le C_n \frac{e^{\Im (z/s-y)}}{(1+|z/s-y|_2)^{n+1}},
	$$
	so that, for $y\in\R^n,z\in\C^n$ and $|\Im z|_2\le s$:
	$$
	\left|K\left(\frac{z}{s}-y\right)\right|\le C_n \frac{e}{(1+|\Re(z/s-y)|_2)^{n+1}}.
	$$
	Hence:
	\beq\label{Ora pro nobis}
	|\f_s(z)|\le \norma{f}_{C^0(\R^n)} C_n e\int_{\R^n} \frac{dy}{(1+|y|_2)^{n+1}}.
	\eeq
	since $z/s$ is fixed and can be eliminated by a simple translation. We finally get the following estimate:
	\beq\label{eq:estim1}
	|\f_s|_s=\sup_{z\in\C^n: \abs{\Im z}_2\leq s}|\f_s(z)|\leq C_1(n)\norma{f}_{C^0(\R^n)},
	\eeq
	with
	$$
	C_1(n):=C_n e\int_{\R^n} \frac{dy}{(1+|y|_2)^{n+1}}<\infty.
	$$

	\vskip3mm
	
	$\bullet$ {\it The periodic case.} Fix now an even function $\Psi:\R^n\to[0,1]$, of class $C^\infty$, with support in the ball 
	$\ov B_1(0,1)$ and define the associate
	kernel $K$ as in (\ref{eq:kernel}).
	
	\vskip1mm
	
	Fix  a ${2\pi}\Z^n$-periodic function $f\in C^\ell(\R^n)$ with $\ell\geq n+1$. Then the Fourier expansion
	$$
	f(\th)=\sum_{k\in\Z^n}\ha f_k e^{ik\cdot \th},\qquad \ha f_k=\frac{1}{(2\pi)^n}\int_{\T^n}f(\varphi)e^{-ik\cdot \varphi}d\varphi,
	$$
	converges normally since, by Lemma \ref{F1} in Appendix \ref{Fourier},  for $k\in\Z^n\setm\{0\}$, there exists a universal constant $\jeanpierre(n,\ell)$ satisfying
	\beq\label{FourierMaj}
	\abs{\ha f_k}\leq \jeanpierre(n,\ell)\frac{||f||_{C^{\floor{\ell}}}}{\norm{k}_\infty^{\floor{\ell}}}
	\eeq
	and $\floor{\ell}\ge n+1$ by hypothesis.
	For $s\in\,]0,1]$, the function 
	$$
	\f_s(\th)=\frac{1}{s^n}\int_{\R^n}K\Big(\frac{\th-\varphi}{s}\Big)f(\varphi)d\varphi
	$$ 
	is well-defined and, by the Fubini interversion theorem:
	$$
	\f_s(\th)=\sum_{k\in\Z^n}\ha  f_k\int_{\R^n}K(\varphi)e^{ik\cdot(\th-s\varphi)}d\varphi=\sum_{k\in\Z^n}\ha  f_k e^{ik\cdot\th}\int_{\R^n}K(\varphi)e^{-i sk\cdot \varphi}d\varphi.
	$$
	Hence, since $K$ is the inverse Fourier transform of $\Psi$,
	by the Fourier inversion theorem:
	\beq\label{Jackson polynomial}
	\f_s(\th)=\sum_{k\in\Z^n}\ha  f_k \Psi(sk)\,e^{ik\cdot\th},\quad \th\in\R^n.
	\eeq
	As in the non-periodic case, this makes apparent that $\f_s$ is a continuous truncation of the Fourier expansion of $f$ with 
	a $\Psi$-dependent modification of its Fourier coefficients (the so-called Jackson polynomial):
	\beq
	\ha{(\f_s)}_k=\Psi(sk)\ha  f_k\ . 
	\eeq
	Consequently, the Fourier norm
	$$
	\norma{\f_s}_s=\sum_{k\in\Z^n}\abs{\ha {(\f_s)}_k}e^{s\norm{k}_1}
	$$
	depends only on the harmonics such that $|k|_1\le1/s$ and satisfies
	%
	%
	$$
	\norma{\f_s}_s\leq \sum_{\norm{k}_1\leq 1/s}\abs{\ha {(\f_s)}_k}\,e^{s\norm{k}_1} \leq e\sum_{\norm{k}_1\leq 1/s}\abs{\ha {(\f_s)}_k}\leq e\,\sum_{k\in\Z^n}\abs{\ha  f_k}.
	$$
	Hence, by (\ref{FourierMaj}):
	\beq\label{eq:majper}
	\norma{\f_s}_s\leq C_2(\ell) \norm{f}_{C^{\floor{\ell}}}
	\eeq
	with
	\beq\label{eq:constper}
	C_2(\ell):=e\Bigg(1+\jeanpierre(n,\ell)\sum_{k\in\Z^n\setm\{0\}}\frac{1}{\norm{k}_\infty^{[\ell]}} \Bigg)
	\eeq
	
	\vskip3mm
	\def\F{{\bf F}}
	$\bullet$ {\it Functions on $\A^n$.} We finally gather together the previous two cases. Let $\Phi\otimes\Psi:\R^n\times\R^n\to[0,1]$ be defined by 
	$$
	\Phi\otimes\Psi(x,\th)=\Phi(x)\Psi(\th),
	$$
	and define the kernel
	$$
	K(y,\varphi)=\int_{\R^{2n}}\Phi\otimes\Psi(x,\th)\,e^{-i (x,\th)\cdot (y,\varphi)}\,dxd\th=K_\Phi(y)K_\Psi(\varphi)=K_\Phi\otimes K_\Psi(y,\varphi)
	$$
	where $K_\Phi$ and $K_\Psi$ are defined as above.
	\vskip2mm
	Fix a function $f:\R^n\times\R^n\to\C$, $2\pi\Z^n$-periodic with respect to its last $n$ variables,  with support in $\ov B_2(0,R_0)\times\R^n$ for some $R_0>0$, belonging to $C^\ell(\R^{2n})$ with $\floor{\ell}\geq n+1$. For $s\in\,]0,1]$ and $(x,\teta)\in\R^n\times\R^n$, set
	$$
	\begin{array}{lll}
	\f_s(x,\th)&\!\!\!=\dsp\int_{\R^{2n}}K(y,\varphi)f(x-sy,\th-s\varphi)dyd\varphi\\[9pt]
	&\!\!\!=\dsp\int_{\R^{2n}}K(y,\varphi)\sum_{k\in\Z^n} \ha f_k(x-sy)e^{ik\cdot(\th-s\varphi)}dyd\varphi\\[9pt]
	\end{array}
	$$
	with
	\beq\label{eq:Fourier}
	\ha f_k(u)=\frac{1}{(2\pi)^n}\int_{\T^n}f(u,v) e^{-i k\cdot v} dv.
	\eeq
	Note that $f_k$ is $C^{\ell}$, with support in $\ov B_2(0,R_0)$, so that the previous study on the non-periodic case applies to $f_k$.
	\vskip2mm
	By Fubini interversion
	\beq\label{Jackson_polynomial}
	\begin{array}{lll}
		\f_s(x,\th)&\!\!\!=\dsp\sum_{k\in\Z^n}\int_{\R^{2n}}K(y,\varphi)\ha f_k(x-sy)e^{ik\cdot(\th-s\varphi)}dyd\varphi\\[9pt]
		&\!\!\!=\dsp\sum_{k\in\Z^n}\Big(\int_{\R^{n}}K_\Phi(y)\ha f_k(x-sy)dy\Big)\Big(\int_{\R^{n}}K_\Psi(\varphi)e^{ik\cdot(\th-s\varphi)}d\varphi\Big) \\[12pt]
		&\!\!\!=\dst\sum_{k\in\Z^n}({\bf \ha f}_k)_s(x)\Psi(sk)e^{ik\cdot\th}\\[10pt]
	\end{array}
	\eeq
	where $({\bf \ha f}_k)_s$ stands for the analytic smoothing of the Fourier coefficient $\ha f_k$. This proves that the Fourier coefficient $(\ha{{\bf f}}_s)_k(x)$
	relative to the periodic variable $\th$ reads
	\beq\label{eq:Fouriercoeff}
	(\ha \f_s)_k(x)=\Psi(sk)({\bf \ha f}_k)_s(x),\quad k\in\Z^n.
	\eeq
	Expressions \eqref{Jackson_polynomial} and \eqref{eq:Fouriercoeff} make clear that the whole smoothing procedure of a function depending both on action and angle variables consists in constructing a Jackson trigonometric polynomial by smoothing the Fourier coefficients and by suitably truncating the Fourier series.
	
	\def\fk{{\bf f_k}}
	\vskip2mm
	
	Using the definition of $\Psi$, $(\ha \f_s)_k=0$ when $\abs{k}_1>1/s$ and,  by \eqref{eq:Fouriercoeff} and \eqref{eq:estim1}:
	\beq\label{adeste fideles}
	|(\ha \f_s)_k(z)|\le |(\ha {\bf f}_k)_s(z)|\le C_1(n) \norma{\ha f_k}_{C^0(\R^n)}\le C_1(n)\jeanpierre(n,\ell)\frac{|f|_{C^{\floor{\ell}}(\R^n)}}{|k|_\infty^{\floor{\ell}}},\quad k\neq 0, \ \abs{k}_1\leq 1/s,
	\eeq
	and
	\beq\label{Alleluja}
	|(\ha \f_s)_0(z)|\le C_1(n)\norma{\ha f_0}_{C^0(\R^n)}\le C_1(n)\norma{f}_{C^0(\R^n)}.
	\eeq
	As for the weighted Fourier norm of ${\bf f}_s$, we finally get:
	$$
	\begin{array}{lll}
	||\f_s||_{s,s}&=&\sup_{\norm{\Im z}_2\leq s}\sum_{k\in\Z^n} \abs{ (\ha\f_s)_k(z)} \,e^{s\norm{k}_1}\\[8pt]
	&\leq&C_1(n)\norma{f}_{C^0(\R^n)}+\dsp\sum_{\substack{k\in\Z^n\backslash\{0\}:\\|k|_1\le 1/s}} eC_1(n)\jeanpierre(n,\ell)\frac{|f|_{C^{\floor{\ell}}(\R^n)}}{|k|_\infty^{\floor{\ell}}}\le C_L(n,\ell) |f|_{C^\ell(\R^n)}\,,
	\end{array}
	$$
	where 
	\beq
	C_L(n,\ell):=C_1(n)\left(1+e\jeanpierre(n,\ell)\sum_{k\in\Z^n}\frac{1}{|k|_\infty^{\floor{\ell}}}\right)<+\infty.
	\eeq


	\subsection{The main result with an application to normal forms.}\label{strategia parchetto}


	\subsubsection{Main result} Gathering together the elements of the previous section, we get the following result.

	\begin{thm}[Analytic smoothing]\label{th_an_smooth} Fix an integer $n\geq 1$, $R>0$ and $s\in\,]0,1]$.
		Let $f$ be a $C^\ell$ function on $B_\infty(0,2R)\times \T^n$.
		Then
		there exist two constants $\jess(R,\ell,n),\santi(R,\ell,n)$ and an analytic function $\tf_s$
		on the set $\A_s^n$ satisfying
		
		\begin{equation}\label{est_smoothing}
		\norma{f-\smooth{f}}_{C^p(\pallaro)} \le
		\  \jess(R,\ell,n)\ s^{\ell-p} \normcl{f}{\ell}{B_\infty(0,2R)\times\T^n}\quad \text{ for any  integer $0\le p\le \floor{\ell}$}
		\end{equation}
		and
		\begin{equation}\label{fourier_est_lissage}
		\norma{\tf_s}_{s,s}\le \santi(R,\ell,n) {\normcl{f}{\ell}{B_\infty(0,2R)\times\T^n}}.
		\end{equation}
		Moreover, $\smooth{f}$ is a trigonometric polynomial in the angular variables. 
	\end{thm}
	\begin{proof}
		Fix a function $\chi\in C^\infty(\R^n)$, with values in $[0,1]$, equal to $1$ on the ball $\pallara$ and with support in $\pallaradue$.
		Then the product $\ov f:=\chi f$ is $C^\ell$ on $\A^n$, has compact support in $\pallaradue \times \T^n$ and coincides with $f$ on $ \pallara \times \T^n$.
		Moreover
		$$
		\normcl{\ov f}{\ell}{B_\infty(0,2R)\times\T^n}\leq {\tC_K}\normcl{f}{\ell}{B_\infty(0,2R)\times\T^n}
		$$
		where ${\tC_K}=C\normcl{\chi}{\ell}{\pallaro}$ and $C$ is a universal constant. By the Jackson-Moser-Zehnder  theorem applied to 
		$\ov f$,  there is an analytic function $\smooth{\bar{f}}$ on $\A^n_s$ satisfying 
		\begin{equation}\label{chierchia}
		\left|\partial^\alpha \smooth{\bar{f}}(I,\teta)-\sum_{\substack{\beta\in\mathbb{N}^{2n}:\\|\beta|\leq \floor{\ell}-|\alpha|}}\partial^{\alpha+\beta} \bar{f}(\Re  (I,\teta))\frac{(\Im (I,\teta))^\beta}{\beta!}\right|\le \jack  s^{{\ell}-|\alpha|}\normcl{\bar{f}}{\ell}{\A^n},
		\end{equation}
		so that for any $p\le \floor{\ell}$:
		\begin{equation}\label{chierchia_2}
		\norma{\bar{f}-\smooth{\bar{f}}}_{C^p(\A^n)} \le 
		\  \jack s^{\ell-p} \normcl{\bar{f}}{\ell}{\A^n}.
		\end{equation}
		As a consequence, taking the form of $\chi$ into account, one gets
		\begin{equation}
		\norma{{f}-\smooth{f}\ }_{C^p(\pallaro)} \le {\tC_K} \jack  s^{\ell - p} \normcl{f}{\ell}{B_\infty(0,2R)\times\T^n}.
		\end{equation}
		Setting $\jess:={\tC_K}  \jack$ and, since the analyticity width $\rho$ of the integrable part $h$ is greater than $s$, the bound \eqref{est_smoothing} follows.
		The proof of (\ref{fourier_est_lissage}) is an immediate consequence of the previous paragraphs if one sets $\santi:=C_L\times C_K$.
	\end{proof}
	
	\def\bH{{\bf H}}
	\def\bh{{\bf h}}

	
	\subsubsection{An easy way to derive normal forms for H\"older functions from analytic ones.}\label{ssec:easy}
Let us now explain our strategy for a general H\"older Hamiltonian, {we will then restrict ourselves to the case where $h$ is analytic}.	Let
	\begin{equation}
	H(I,\teta):=h(I)+f(I,\teta)
	\end{equation}
	be $C^\ell$ on $B_\infty(0,2R) \times \T^n$. Given $s\in\,]0,1]$, let $\bH_s$ be the $s$-smoothed analytic function given by Theorem \ref{th_an_smooth}
	applied to the function $H$.
	By classical constructions (alluded to in the introduction and which will be recalled in the following), there exist (close to identity) symplectic analytic local diffeomorphisms 
	$\Psi$ defined on domains $D\subset \A^n$ which bring $\bH_s=\bh_s+\f_s$ to the normal form $\bH_s\circ\Psi: D\to \R$:
	\beq
	\bH_s\circ\Psi=\bh_s+\bf{g}+ \f_s^*
	\eeq
	where $\bh_s$ is nothing else than the smoothed initial integrable Hamiltonian, $\tg$ is a resonant part which controls the fast drift in certain directions
	and $\f_s^*$ is a very small remainder -- all these functions being analytic on $D$. The keypoint in our subsequent constructions is the following very 
	simple equality
	\beq\label{eq:smoothedNF}
	H\circ\Psi=\bH_s\circ\Psi+(H-\bH_s)\circ\Psi=\bh_s+\bf{g}+ \big[\f_s^*+(H-\bH_s)\circ \Psi\big].
	\eeq
	This is a normal form for $H$, obtained by composition of $H$
	with an {\it analytic} diffeomorphism, in which the first three terms are analytic on $D$ and only the last one is $C^\ell$.
	So $H\circ\Psi$ has the same structure and dynamical interpretation as $\bH_s\circ \Psi$, {\it provided that the $C^\ell$ size of the additional
		remainder $(H-\bH_s)\circ \Psi$ is of the same order as the size of the initial remainder $\f_s^*$}. This issue strongly depends on the analytic smoothing
	method in use, we will show in the sequel that the Jackson-Moser-Zehnder method is relevant for our purposes.
	Our study will be even easier since we assume from the beginning that the integrable part $h$ is analytic.
	
	\vskip2mm
	
	It turns out that the same smoothing method - and the same simple way to get a normal form from an analytical one -  are also relevant in many other functional classes, the main ones being the Gevrey classes already used in \cite{Marco_Sauzin_2003}, but also other ultradifferentiable ones. This will be developed in a further work.
	
	\section{Estimates of stability}\label{steep case}
	The aim of this section is to prove Theorem \ref{MainTh3}. The proof consists of several steps. Following the discussion in section \ref{aiuto} of the introduction, we first build an appropriate resonant covering of the phase space for the integrable Hamiltonian $h$. Secondly, we study the local dynamics by applying P\"{o}schel's resonant normal form (see Appendix \ref{loch_app}) in each resonant block and we set the dependencies of the ultraviolet cut-off $K$ and analyticity widths $r,s$ on the perturbative parameter $\eps$. Finally, we exploit the properties of the resonant covering and we obtain a global result of stability by exploiting the so called "capture in resonance" argument.

	\subsection{Construction of the resonant patchwork}\label{geometria}
	In the sequel, we follow ref. \cite{Guzzo_Chierchia_Benettin_2016}, in which the choices of the parameters and the dependencies of the small denominators on the ultraviolet cut-off $K$ are justified heuristically. 	
	 For the sake of clarity,  in order to have coherent notations  we denote by $D_\Lambda$ rather than $\cB_\Lambda$ the resonant blocks introduced in Section \ref{pippo}, moreover  when possible we will not keep track of constants \footnote{i.e. of quantities depending only on the fixed parameters of the problem, namely $n,h,\ell$ and on the indices of steepness $\alg_1,...,\alg_{n-1}$.} but rather indicate their presence in bounds and equalities by using the following symbols respectively: $\circeq ,\lessdot$ and $\gtrdot$.
	
	We start by setting some parameters, depending on the steepness indices $\alg_1,..,\alg_{n-1}$ of $h$, that will be useful throughout this section. \begin{equation}\label{parametrini}
	p_j:=\begin{cases}
	\Pi_{i=j}^{n-2}\alg_i  &, \quad \text{if }j\in\{1,...,n-2\}\\
	1 &, \quad \text{if }j\in\{n-1,n\}
	\end{cases}\ ;
	\quad q_j:=np_j-j\ , \ \ j\in\{1,...,n\}\ ; \quad c_j:=q_j-q_{j+1} \ , \ \ j\in\{1,...,n-1\}
	\end{equation}
	and set
	\begin{equation}\label{ab}
	a:=\frac{1}{2n\alg_1...\alg_{n-2}}=\frac{1}{2np_1} \quad , \qquad b:=\frac{1}{2n\alg_1...\alg_{n-1}}=\frac{a}{\alg_{n-1}}\quad , \qquad \tR(\eps):\circeq\, \eps^b\ .
	\end{equation}
	With this setting, we fix an action $I_0\in B_\infty(0,R/4)$ and we consider its neighborhood $B_2(I_0,\tR(\eps))$. 
	
	Since $h$ is steep in $B_\infty(0,R)$, the norm of the frequency $\omega:=\partial_I h(I)$ at any point of this set admits a uniform lower positive bound, that is $\inf_{I\in B_\infty(0,R)}||\omega(I)||\gtrdot 1$. Hence,  when studying the geography of resonances for $h$, for sufficiently small $\eps$ and without any loss of generality we can just consider maximal lattices $\Lambda\subset \Z^n_K$ of dimension $j\in\{0,...,n-1\}$, with $K\ge 1$ the ultraviolet cut-off. For a lattice $\Lambda$ of dimension $j\in\{0,...,n-1\}$ we define its associated {\it resonant zone} as
	\begin{equation}\label{resonant_zone}
	Z_\Lambda:=\{I\in B_2(I_0,\tR(\eps)):\ \forall k\in \Lambda\ \text{one has }\ \ |k\cdot \omega(I)|<\delta_\Lambda\}\ ,\ \ \delta_\Lambda:\circeq \frac{1}{|\Lambda|K^{q_j}}\ .
	\end{equation}
and its associated {\it resonant block} $D_\Lambda$ as
	\begin{equation}\label{blocchetti}
	D_\Lambda:=Z_\Lambda\backslash \bigcup_{\Lambda':\,\dim \Lambda'=j+1}Z_{\Lambda'}\ .
	\end{equation}
	Note that $D_\Lambda$ corresponds to that part of the resonant zone $Z_\Lambda$ which does not contain any other resonances other than the one associated to $\Lambda$.
	In particular, this implies that for the completely non-resonant block associated to $\Lambda=\{0\}$ and for any block $\Lambda$ corresponding to a maximal resonance of dimension $j=n-1$ one has, respectively 
	\begin{equation}\label{maledizione}
	D_0:=B(I_0,\tR(\eps))\backslash \bigcup_{\Lambda':\ \dim \Lambda'=1}Z_{\Lambda'}\quad  \text{ and } \quad D_\Lambda=Z_\Lambda\ .
	\end{equation}
	For any $j\in\{0,...,n-1\}$ we set
	\begin{equation}
	D_j:=\bigcup_{\Lambda:\ \dim \Lambda=j}D_{\Lambda}\quad,\qquad Z_j:= \bigcup_{\Lambda:\ \dim \Lambda=j}Z_{\Lambda} \ .
	\end{equation}
	It is easy to see from \eqref{blocchetti} that
	\begin{equation}
	D_j=Z_j\backslash Z_{j+1}
	\end{equation}
	so that from the definition of $D_0$ in \eqref{maledizione} one has the decompositions
	\begin{equation}\label{covering}
	B_2(I_0,\tR(\eps))=\bigcup_{i=0}^{n-1} 
	D_i\quad ,\qquad B_2(I_0,\tR(\eps))=\left(\bigcup_{i=0}^{j-1} 
	D_i\right)  \cup Z_j \qquad \forall j=1,...,n-1\ .
	\end{equation}
	As we have explained in the introduction (see section \ref{aiuto}), a large drift over a short time of any action variable $I\in D_\Lambda$ is only possible along the plane of fast drift $I+\langle \Lambda\rangle$ spanned by the vectors belonging to $\Lambda$. Moreover, the fast motion of the orbit starting at $I$ along $I+\langle\Lambda\rangle$ can take the actions out of the block $D_\Lambda$. So, we are interested in understanding what happens when the actions leave $D_\Lambda$ but keep staying in $Z_\Lambda$.  Hence, we are naturally taken to consider the intersection of a neighborhood of $I+\langle\Lambda\rangle$ with $Z_\Lambda$. In this spirit, we fix 
	\begin{equation}\label{rhoerre}
	\rho(\eps):=\displaystyle\frac{\tR(\eps)}{2n}
	\end{equation}
	and, for any $0<\eta\le \rho(\eps)$ and for any action $I\in D_\Lambda$ with $\Lambda\neq\{0\}$, we define the {\it disc} associated to $I$ as
	\begin{equation}\label{disc}
	\mathbf{D}^{\rho}_{\Lambda,\eta}(I):=\Bigg(\ \bigg(\bigcup_{I'\in I+\langle\Lambda\rangle}B_2(I',\eta)\bigg)\cap Z_\Lambda \cap B\big(I_0,\tR(\eps)-\rho(\eps)\big) \  \Bigg)_I
	\end{equation}
	where the subscript $I$ denotes the connected component of the set containing the action $I$. Since we are going to study the fate of all orbits starting at a fixed block $D_\Lambda$, with $\Lambda\neq \{0\}$, that exit such block in a short time along the plane of fast drift, we are also led to define the {\it extended resonant block} 
	\begin{equation}\label{blocco_esteso}
	D_{\Lambda,r_\Lambda}^\rho:=\Bigg(\bigcup_{I\in D_\Lambda\cap B(I_0,\tR(\eps)-\rho(\eps))}\mathbf{D}^{\rho}_{\Lambda,r_\Lambda}(I)\Bigg)\ \subset\  Z_\Lambda\cap B\big(I_0,\tR(\eps)-\rho(\eps)\big)\quad,\qquad r_\Lambda:= \frac{\delta_\Lambda}{M}\ ,
	\end{equation}
	where $M$ was defined in $\eqref{M}$. 
	In the same way, the {\it extended non-resonant block} is defined as 
	\begin{equation}\label{blocco_esteso_nr}
	D_0^\rho:=D_0\cap B(I_0,\tR(\eps)-\rho(\eps))\ .
	\end{equation}

	\subsection{The resonant blocks}
 As we have explained there, Nekhoroshev proved in \cite{Nekhoroshev_1973} that, if $h$ is steep, when any action $I\in D_\Lambda$, with $\Lambda\neq\{0\}$, moves along the plane of fast drift, it must exit the resonant zone $Z_\Lambda$ after having travelled for a short distance. Indeed, if $h$ is steep with steepness indices $\alg_1,...,\alg_{n-1}$ one can prove that the diameter of the intersection of a neighborhood of the fast drift plane with the resonant zone is small in the sense given by the following 
	\begin{lemma}\label{GCB1}
		For any $\Lambda\neq 0$, $\dim\Lambda=j\in\{1,...,n-1\}$, for any $I\in D_\Lambda \cap B(I_0,\tR(\eps)-\rho(\eps))$ and for any $I'\in \mathbf{D}^\rho_{\Lambda,r_\Lambda}(I)$ one has
		\begin{equation}
		\abs{I-I'}_2\le r_j\quad ,\qquad \text{ where } \quad r_j:\circeq  \displaystyle\frac{1}{K^{q_j/\alpha_j}}\ .
		\end{equation}
	\end{lemma}
For a proof of this result we refer to Lemma 2.1 of ref. \cite{Guzzo_Chierchia_Benettin_2016}.

	We notice that a smaller value of $\eps$, i.e. a higher value of $K$ since the ultraviolet cut-off is always a decreasing function of $\eps$, leads to a closer maximal distance between any action $I$ belonging to a resonant block and any action belonging to its disc.  
	
	Since we will perform  normal forms in the (extended) resonant blocks, we also need an estimate of the small divisors in these sets, namely we have
	\begin{lemma}\label{GCB2}
		For any maximal lattice $\Lambda\in\Z^n_K$ of dimension $j\in\{0,...,n-1\}$, for any $ k\in\Z^n_K\backslash\Lambda$ and for any $  I\in D^\rho_{\Lambda,r_\Lambda}$ one has 
		\begin{equation}\label{small_div_r_steep}
		|\langle k,\omega(I)\rangle |\ge \alpha_\Lambda:\circeq\displaystyle \frac{1}{|\Lambda|K^{q_j-c_j}}\ ,
		\end{equation}
		whereas for any action $I$ in the completely non-resonant block $D_0$ and for any $k\in\Z^n_K$ one has
		\begin{equation}\label{small_div_nr_steep}
		|\langle k,\omega(I)\rangle |\ge \alpha_0:\circeq\displaystyle \frac{1}{K^{q_1}}\ .
		\end{equation}
	\end{lemma} 
	We refer again to \cite[Lemma 2.2]{Guzzo_Chierchia_Benettin_2016}  for a proof of this result. 
	
	Finally,  a key ingredient in order to insure stability in the steep case is the fact that, when possibly exiting a resonant zone along the plane of fast drift, the actions must enter another resonant zone associated to a lattice of lower dimension. This is the content of 
	\begin{lemma}\label{GCB3}
		Let $\Lambda,\Lambda'$ two maximal lattices of $\Z^n_K$ having the same dimension $j\in\{1,...,n-1\}$. 
		Then one has
		\begin{equation}
		\text{closure}\left( D^\rho_{\Lambda,r_\Lambda}\right)\cap Z_{\Lambda'}=\varnothing\ .
		\end{equation}
	\end{lemma}
	Once again, the proof of this Lemma can be found in \cite{Guzzo_Chierchia_Benettin_2016} (Lemma 2.3). 
	
	With the ingredients of this paragraph, we are able to prove stability. 
	\subsection{Proof of Theorem \ref{MainTh3}}
	We start by giving the standard estimates of stability in the completely non-resonant extended block $D_0^\rho$. Note that the following bounds do not require any geometric assumption on the integrable part $h$.
	\begin{lemma}[Non-resonant Stability Estimates]\label{nr_stab_steep} For any sufficiently small $\eps$ and for any time $t$ satisfying 
		\begin{equation}\label{t_nr_steep}
		|t|\le T_0:\circeq\frac{1}{(1+a\ell)|\ln\eps|^{\ell-1}\  \eps^{a(\ell-1)+1/2}}\quad ,\qquad a:=\frac{1}{2np_1}\ ,
		\end{equation}
		any initial condition $I(0)\in D_0^\rho$ drifts at most as
		\begin{equation}\label{deltaI nr_steep}
		|I(t)-I(0)|_2\ledot  \eps^{1/2}\ .
		\end{equation}	
	\end{lemma}
	\begin{proof}
		
	Our goal is to apply Pöschel's normal form (see Lemma \ref{poschel normal form}) to the smoothed Hamiltonian of Theorem \ref{th_an_smooth} with analyticity widths $r$ and $s$.  
	
	$\bullet$ {\bf Normal form}
	
	By monotonicity of the Fourier norm w.r.t. the action variables and \eqref{fourier_est_lissage} we immediately get, 
	\begin{equation}\label{epsprimo}
	||\smooth{f}||_{{r,s}}\le ||\smooth{f}||_{{s,s}} \le  \santi(R, \ell, n ) \eps  =:\epsilon\ ,  
	\end{equation}
	for any $r\le s$,  where we set $ \eps:=|f|_{C^\ell(\pallaro)} $.	
	
	Denote
		$$
	\cB_{\varrho,\sigma}:=\{(I,\teta)\in\C^n:|I-B_\infty(0,R/4)|_2<\varrho\ ,\ \  \teta\in\T^n_\sigma\}\ \,,
	$$
	since $h$ is analytic,  we chose not to regularize it further.  So let $\tH_s := h(I) + \tf_s$ be the corresponding analytic Hamiltonian defined on $\cB_{s,s}$.  By Pöschel's Lemma \ref{poschel normal form} applied in the complex extension, denoted $\cD^\rho_{0,r,s}$, of the non-resonant block $D_0^\rho$, with $\varrho' \rightsquigarrow r, \varrho \rightsquigarrow s, \s \rightsquigarrow s$, if
	\begin{equation}\label{soglie}
	\epsilon\ledot\frac{\alpha_0 \,r}{K}\ ,\ \ r\ledot \min\left(\frac{\alpha_0}{K},s\right)\ ,\ \ Ks\ge 6
	\end{equation}
	are satisfied, then there exists a symplectic diffeomorfism $\Psi_0$ that puts $\tH_s$ into resonant normal form: 
	\begin{equation}\label{resnf}
	\tH_s \circ \Psi_0 =  h(I) + \tg + \tf^*_s\ ,\ \ \{h,\tg\}=0,\quad \Psi_0:\cD^\rho_{0,r/2,s/6}\longrightarrow\cD^\rho_{0,r,s}\ .
	\end{equation}
 In particular the resonant and non-resonant part satisfy, respectively,
	\begin{equation}\label{resto}
	||{\tg-\bf g_0}||_{r/2,s/6}\ledot \epsilon
	\ ,\ \  
	||\tf_{s}^*||_{r/2,s/6}\le e^{-\frac{Ks}{6}}\epsilon
	\end{equation}
	where ${\bf g_0}:=P_\Lambda P_K\smooth{f}$ and $P_\Lambda,P_K$ are the projectors defined in Lemma \ref{poschel normal form}.
	
	\vskip2mm

$\bullet$ {\bf Setting of the initial parameters}

\vskip2mm

Let us set the following dependences on $\epsilon$ of the ultraviolet cut-off $K$ and of the analyticity widths $r,s$
		\begin{equation}\label{scegli}
		K:= \left(\frac{\epsilon_0}{\epsilon}\right)^a\ ,\ \ s:\circeq\,\left(\frac{\epsilon}{\epsilon_0}\right)^{a}\left|\ln\left[\left(\frac{\epsilon}{\epsilon_0}\right)^{6(1+a\ell)}\right]\right|\ ,\ \ r:\circeq \frac{1}{K^{1+q_1}}\circeq \left(\frac{\epsilon}{\epsilon_0}\right)^{a(1+q_1)}= \left(\frac{\epsilon}{\epsilon_0}\right)^{1/2}\ .
		\end{equation}
	where $\epsilon_0$ is a free parameter and $\epsilon\le \epsilon_0$ since $K\ge 1$. 

\begin{rmk}\label{pippone}
The freedom in the definitions above is  subordinated to the fact that, in order for the construction to be meaningful,  the reminder produced by the normal form must be less than or equal to the size of the additional term  $(H - \smooth{H})\circ\Psi_0$,  byproduct of the analytic smoothing.  As we are working in finite regularity, the latter is expected to be polynomial.  The reminder of the normal form being of order $e^{-Ks}$,  one must have $Ks \sim O(|\log\epsilon|^c)$ for some $c>0$. Since $s$ tunes the size of the remainder yielded by the analytic smoothing,  it has to be polynomial.  Hence one is left with two possibilities: either the choice we made in \eqref{scegli},  or to set $K \sim \epsilon^{-a} |\log\epsilon|^{c} $ and $s \sim \epsilon^a$. However this second choice would worsen the exponents of stability,  since the thresholds of applicability in the normal form lemma strongly depend on $K$.
Of course, to deal with other regularity classes, such as the Gevrey one,  other choices must be made.
\end{rmk}
	
		By plugging the choices \eqref{scegli} into the three thresholds in \eqref{soglie}, it is easy to see that there exists an appropriate choice of $\epsilon_0$ that makes the three conditions to be simultaneously satisfied. Hence, for the H\"older Hamiltonian $$H=h+f = \tH_s + f - \tf_s,\qquad \tH_s := h + \tf_s$$ we can write
		\begin{equation}\label{effetto}
		H\circ\Psi_0=\smooth{H}\circ\Psi_0+(f-\smooth{f})\circ\Psi_0=h + \tf_s^*+ (f-\smooth{f})\circ\Psi_0 \ .
		\end{equation}
	 Note that since we are in a completely non-resonant block, the resonant term $\tg$ does not appear in the normal form. 
		Now, the normal form in Lemma \ref{poschel normal form} insures that there exists a constant $\xi>1$ such that any initial condition $(I(0),\teta(0))\in D_0^\rho\times\T^n$ is mapped by $\Psi_0$ into $(\tI(0),\vartheta(0))\in (\cD^\rho_{0,\frac{r}{32\xi}})^\R\times\T^n$. For any time $t$ such that the normalized flow $\Phi^t_{H\circ \Psi_0}: (\tI(0),\vartheta(0))\longmapsto(\tI(t),\vartheta(t))$ starting at $(\cD^\rho_{0,\frac{r}{32\xi}})^\R\times\T^n$ does not exit from $(\cD^\rho_{0,r/2})^\R\times\T^n$, the evolution of the normalized variables reads $(i=1,...,n)$
	\begin{align}\label{accroissements_finis}
	\begin{split}
	|\tI_i(t)-\tI_i(0)|&\le\int_0^t \sup_{(\tI,\vartheta)\in (\cD^\rho_{0,\frac{r}{32\xi}})^\R\times \T^n}\bigg(\left|(\partial_{\vartheta_i} \tf^*_s)\circ\Phi^t_{H\circ \Psi_0}\right|+\left|\{\partial_{\vartheta_i} [(f-\smooth{f})\circ\Psi_0]\}\circ\Phi^t_{H\circ \Psi_0}\right|\bigg) dt\\
	& \le \int_0^t\left(\sup_{(\tI,\vartheta)\in (\cD^\rho_{0,r/2})^\R\times \T^n}|\partial_{\vartheta_i} \tf^*_s|+\sup_{(\tI,\vartheta)\in (\cD^\rho_{0,r/2})^\R\times \T^n}|\partial_{\vartheta_i}[(f-\smooth{f})\circ\Psi_0]|\right) dt\\
	&\le|t|\left[ \frac{||\smooth{f}^*||_{r/2,s/6}}{s}+	\norma{f-\smooth{f}}_{C^1(B_\infty(0,R/2)\times\T^n)}\times \sup_{(\tI,\vartheta)\in (\cD^\rho_{0,r/2})^\R\times \T^n}|\partial_{\vartheta_i}\Psi_0|\right]\ .
	\end{split}
	\end{align}
	The normal form Lemma \ref{poschel normal form}, together with the choices in \eqref{scegli} and the definition of $\epsilon$ in \eqref{epsprimo}, assures that 
	\beq\label{resto_analitico}
	||\smooth{f}^*||_{r/2,s/6}\le e^{-Ks/6}\,\epsilon\ledot \exp\left\{\ln\left[\left(\frac{\epsilon}{\epsilon_0}\right)^{1+a\ell}\right]\right\}\epsilon\ledot \eps^{2+a\ell}\ ,
	\eeq
	whereas, by Theorem \ref{th_an_smooth}, we have
	\beq \label{diff}
	\norma{f-\smooth{f}}_{C^1(B_\infty(0,R/2)\times\T^n)}\ledot s^{\ell-1}\eps\ledot \left|\ln\left[\left(\frac{\epsilon}{\epsilon_0}\right)^{6(1+a\ell)}\right]\right|^{\ell-1} \eps^{1+a(\ell-1)}\ledot \left|\ln\left(\eps^{6(1+a\ell)}\right)\right|^{\ell-1} \eps^{1+a(\ell-1)}\ .
	\eeq
	Finally, by writing in the usual way $|\d_{\vartheta_i} \Psi_0|=|\d_{\vartheta_i} (\Psi_0-\id+\id)|$, the Cauchy estimates together with the bounds in \ref{taglia_trasf} imply (since $r\le s$)
	\beq\label{der}
	\sup_{(\tI,\vartheta)\in (\cD^\rho_{0,r/2})^\R\times \T^n}|\partial_{\vartheta_i}\Psi_0|_2\ledot 1+\max\left\{\frac{1}{24\xi},\frac{1}{32\xi}\frac{r}{s}\right\}\ledot 1\ .
	\eeq
	It is easy to see from estimates \eqref{resto_analitico}, \eqref{diff} and \eqref{der} that, in order, the remainder from the analytic smoothing dominates on the one coming from the normal form, namely 
	$$
	\frac{||\smooth{f}^*||_{r/2,s/6}}{s}\ll \norma{f-\smooth{f}}_{C^1(B_\infty(0,R/2)\times\T^n)}\times \sup_{(\tI,\vartheta)\in (\cD^\rho_{0,r/2})^\R\times \T^n}|\partial_{\vartheta_i}\Psi_0|
	$$
	so that finally we can write
	\beq
	|\tI(t)-\tI(0)|_2\ledot |t|\left|\ln\left(\eps^{6(1+a\ell)}\right)\right|^{\ell-1} \eps^{1+a(\ell-1)}\ .
	\eeq
	Hence, over a time 
	$$
	|t|\ledot \frac{r}{\left|\ln\left(\eps^{6(1+a\ell)}\right)\right|^{\ell-1} \eps^{1+a(\ell-1)}}\ledot  \frac{1}{\left|\ln\left(\eps^{6(1+a\ell)}\right)\right|^{\ell-1} \eps^{1/2+a(\ell-1)}}
	$$
	one has $|\tI(t)-\tI(0)|_2\ledot r$ and, by scaling back to the original variables,
	$$
	|I(t)-I(0)|_2\ledot r\ledot \eps^{1/2}\ .
	$$	
	\end{proof}	
	As for the dynamics in the resonant blocks, we have the following
	\begin{lemma}\label{r_stab_steep}
		Consider a maximal lattice $\Lambda\subset \Z^n_K$ of dimension $j\in\{1,...,n-1\}$.  There exists $\mathtt{T}_j>0$ such that for any sufficiently small $\eps$ and for any initial condition $(I(0),\teta(0))\in \bigg(D_\Lambda\cap B\big(I_0,\tR(\eps)-(j+1)\rho(\eps)\big)\bigg) \times\T^n$,  if one sets
		\begin{align}\label{Tj}
		T_\Lambda:= & \mathtt{T}_j\times \frac{r_\Lambda}{|\ln\eps^{6(1+a\ell)}|^{\ell-1}\,\eps^{1+a(\ell-1)}}\quad,\qquad a:=\frac{1}{2np_1}\ ,
		\end{align}
	 and considers the time of escape of the flow generated by $H$ from the extended resonant block
		\begin{align}
		\tau_e:= & \inf\left\{t\in\R: \Phi^t_{H}\bigg(D_\Lambda\cap B\big(I_0,\tR(\eps)-(j+1)\rho(\eps)\big)\times\T^n\bigg)\not\subset D^\rho_{\Lambda,r_\Lambda}\times\T^n\right\}\ ,
		\end{align}
		the following dichotomy applies:
		\begin{enumerate}
			\item If $|\tau_e|\ge T_\Lambda$ one has
			\begin{equation}
			|I(t)-I(0)|_2< \rho(\eps) 
			\end{equation} 
			over a time $|t|\le T_\Lambda$;
			\item If $|\tau_e|< T_\Lambda$ there exists $i\in\{0,...,j-1\}$ such that 
			$$
			I(\tau_e)\in D_i\cap \bigg(B\big(I_0,\tR(\eps)-j\rho(\eps)\big)\bigg)\ .
			$$
		\end{enumerate}
	\end{lemma}
	\begin{proof}
		We start by considering the case $|\tau_e|\ge T_\Lambda$. In a similar way to what we did in the proof of Lemma \ref{nr_stab_steep}, we apply Pöschel's Normal Form (see Lemma \ref{poschel normal form}) to the smoothed Hamiltonian $\tH_s$ in the complex extension $(D^\rho_{\Lambda,r_\Lambda})_{r_\Lambda}$ of the real extended resonant block $D^\rho_{\Lambda,r_\Lambda}$, with parameters
		\begin{equation}\label{bidibibodibibu}
		K:= \left(\frac{\epsilon_0}{\epsilon}\right)^a\ ,\ \ s:\circeq\,\left(\frac{\epsilon}{\epsilon_0}\right)^{a}\left|\ln\left[\left(\frac{\epsilon}{\epsilon_0}\right)^{6(1+a\ell)}\right]\right|\ ,\ \ r_\Lambda:\circeq \frac{1}{|\Lambda|K^{q_j}}
		\end{equation}
		and with a small divisor estimate given by formula \eqref{small_div_r_steep} in Lemma \ref{GCB2}, namely
		\begin{equation}\label{buongiornissimo}
		\alpha_\Lambda:\circeq\frac{1}{|\Lambda|K^{q_j-c_j}}\ .
		\end{equation}
		As before, we plug \eqref{bidibibodibibu} and \eqref{buongiornissimo} into Pöschel's thresholds \eqref{soglie_Poschel} -- \eqref{soglie_Poschel2} and we derive the conditions
		\begin{align}
			\begin{split}
			\epsilon\ledot \frac{\alpha_\Lambda r_\Lambda}{K} \quad & \longleftrightarrow \quad \left(\frac{\epsilon}{\epsilon_0}\right)^{1-an(p_j+p_{j+1})}\ledot 1\qquad j\in\{1,\ldots,n-1\}\\
			r_\Lambda\ledot \frac{\alpha_\Lambda}{K}\quad & \longleftrightarrow \quad \left(\frac{\epsilon}{\epsilon_0}\right)^{an(p_j-p_{j+1})}\ledot 1 \qquad j\in\{1,\ldots,n-1\}\\
			Ks\ge 6\quad  & \longleftrightarrow\quad 
			\left|\ln\left[\left(\frac{\epsilon}{\epsilon_0}\right)^{6(1+a\ell)}\right]\right|\ge 6\ .
			\end{split}
		\end{align}
		By definition of the parameters $p_j$ in \eqref{parametrini}, it is easy to see that the first two conditions are always satisfied by appropriately choosing $\epsilon_0$, whereas the last condition is trivial. 
		
		Therefore, by taking into account the notations in \eqref{dominos}, there exists a symplectic transformation $\Psi_\Lambda:(D^\rho_{\Lambda,r_\Lambda})_{r_\Lambda/2}\times\T^n_{s/6}\longrightarrow (D^\rho_{\Lambda,r_\Lambda})_{r_\Lambda}\times\T^n_{s}\ ,\ \ (\tI,\vartheta)\longmapsto(I,\teta)$, that takes $H$ into the resonant normal form
		\begin{equation}
		H\circ\Psi_\Lambda=\smooth{H}\circ\Psi_\Lambda+(H-\smooth{H})\circ\Psi_\Lambda=h + \tg+\tf_s^*+ (f-\smooth{f})\circ\Psi_\Lambda
		\end{equation} 
		with $\{h,\tg\}=0$, $||\tf_s^*||_{r/2,s/6}\ledot e^{-Ks/6}\ \eps$. 
		
		Now, for any time $t$ such that $|t|\le T_\Lambda\le |\tau_e|$, the dynamics on the subspace orthogonal to the plane of fast drift $\langle\Lambda\rangle$ can be controlled in the usual way by exploiting the smallness of the non-resonant remainder $\tf_s^*$, as well as that of $(f-\tf_s)\circ\Psi_{\Lambda}$. Namely, for any initial position in the actions $I(0)\in D_\Lambda$, by the first estimate in \eqref{taglia_trasf} one has that the associated normalized coordinate satisfies $\tI(0)\in (D_\Lambda)^\R_{\frac{r_\Lambda}{32\xi}}$, where $(D_\Lambda)^\R_{\frac{r_\Lambda}{32\xi}}$ represents the real projection of the complex extension of width $\frac{r_\Lambda}{32\xi}$ around $D_\Lambda$ (not to be confused with the extended resonant block) and where $\xi>1$ is a free parameter that can be suitably adjusted. By taking into account the fact that $\Pi_{\langle\Lambda\rangle^\perp}(\partial_{\vartheta} \tg)=0$, one can write 
			\begin{align}\label{accroissements_finis_2}
		\begin{split}
		&\left|\Pi_{\langle\Lambda\rangle^\perp}\big(\tI(t)-\tI(0)\big)\right|_2\\
		& \le \int_0^t \sup_{(\tI,\vartheta)\in \left(D_\Lambda\right)_{\frac{r_\Lambda}{32\xi}}^\R\times \T^n}\bigg(\left|\Pi_{\langle\Lambda\rangle^\perp}(\partial_{\vartheta} \tg+\partial_{\vartheta} \tf^*_s)\circ\Phi^t_{H\circ \Psi_\Lambda}\right|_2+\left|\Pi_{\langle\Lambda\rangle^\perp}\{\partial_{\vartheta} [(f-\smooth{f})\circ\Psi_\Lambda]\}\circ\Phi^t_{H\circ \Psi_\Lambda}\right|_2\bigg) dt\\
		& \le \int_0^t \sup_{(\tI,\vartheta)\in \left(D_\Lambda\right)_{\frac{r_\Lambda}{32\xi}}^\R\times \T^n}\bigg(\left|(\partial_{\vartheta} \tf^*_s)\circ\Phi^t_{H\circ \Psi_\Lambda}\right|_2+\left|\{\partial_{\vartheta} [(f-\smooth{f})\circ\Psi_\Lambda]\}\circ\Phi^t_{H\circ \Psi_\Lambda}\right|_2\bigg) dt\\
		& \le  \sup_{(\tI,\vartheta)\in \left(D^\rho_{\Lambda,r_\Lambda}\right)_{\frac{r_\Lambda}{32\xi}}^\R\times \T^n}\bigg(\left|(\partial_{\vartheta} \tf^*_s)\right|_2+\left|\{\partial_{\vartheta} [(f-\smooth{f})\circ\Psi_\Lambda]\}\right|_2\bigg) |t|\ ,\\
		\end{split}
		\end{align}
		where the last inequality follows from the fact that $|t|\le\tau_e$ and, since the initial variables are confined in $D^\rho_{\Lambda,r_\Lambda}$, the normalized ones stay in $(D^\rho_{\Lambda,r_\Lambda})^\R_{\frac{r_\Lambda}{32\xi}}$ over the same time. 
		
		Since $|t|\le T_\Lambda \le \tau_e$, by the same arguments that were used in estimate \eqref{accroissements_finis} and estimate \eqref{accroissements_finis_2} we obtain
		\begin{align}\label{quarto}
		\begin{split}
			\left|\Pi_{\langle\Lambda\rangle^\perp}\big(\tI(t)-\tI(0)\big)\right|_2\ledot &\, |t|\left|\ln\left(\eps^{6(1+a\ell)}\right)\right|^{\ell-1} \eps^{1+a(\ell-1)}\\
			\ledot &\, \mathtt{T}_j\times \frac{r_\Lambda}{\left|\ln\left(\eps^{6(1+a\ell)}\right)\right|^{\ell-1} \eps^{1+a(\ell-1)}}\left|\ln\left(\eps^{6(1+a\ell)}\right)\right|^{\ell-1} \eps^{1+a(\ell-1)}= \frac{r_\Lambda}{4}
			\end{split}
		\end{align}
		by suitably choosing $\mathtt{T}_j$.
		
		Let us decompose the variation of the action variables as
		\begin{align}
		\begin{split}
		I(t)-I(0)=&I(t)-\tI(t)+\tI(t)-\tI(0)+\tI(0)-I(0)\\
		=& I(t)-\tI(t)+\Pi_{\langle\Lambda\rangle^\perp}\big(\tI(t)-\tI(0)\big)+\Pi_{\langle\Lambda\rangle}\big(\tI(t)-\tI(0)\big)+\tI(0)-I(0)\ ,
		\end{split}
		\end{align}
		so that estimate \eqref{quarto}, together with the size of the normal form, implies that, for $|t|\le T_\Lambda$, the motion orthogonal to the fast drift plane is bounded by 
		\begin{align}\label{Piemonte}
		\begin{split}
		|I(t)-I(0)-\Pi_{\langle\Lambda\rangle}\big(\tI(t)-\tI(0)\big)|_2\le & |I(t)-\tI(t)|_2+|\Pi_{\langle\Lambda\rangle^\perp}\big(\tI(t)-\tI(0)\big)|_2+|\tI(0)-I(0)|_2\\
		\le & \frac{r_\Lambda}{32\xi} + \frac{r_\Lambda}{4} + \frac{r_\Lambda}{32\xi}\le \frac34 r_\Lambda\ ,
		\end{split}
		\end{align}
		where we have used the fact that $\xi>1$. Hence, by \eqref{Piemonte}, $I(t)\in D^\rho_{\Lambda,r_\Lambda}$ since $I(0)\in D_\Lambda $ and the orbit lies entirely in this set for any $|t|\le T_\Lambda\le \tau_e$; moreover, the definition in \eqref{disc} implies
		$$
		I(t)\in\mathbf{D}^\rho_{\Lambda,\frac34 r_\Lambda}(I(0))\subset \mathbf{D}^\rho_{\Lambda, r_\Lambda}(I(0))\ .$$ 
		This fact, together with Lemma \ref{GCB1}, yields
		\begin{equation}
		|I(t)-I(0)|_2\le r_j \quad , \qquad \text{where} \qquad r_j:\circeq\  \displaystyle\frac{1}{K^{q_j/\alpha_j}}\ , \ \ r_j\gedot r_\Lambda\ .
		\end{equation}
		As it is shown in \cite{Guzzo_Chierchia_Benettin_2016} (formula (38)), a careful choice of the constants leads to
		$$
		\max_{j\in\{1,...,n-1\}}r_j< \rho(\eps)\ ,
		$$ 
		which concludes the proof of the first claim of this Lemma. 
		
		We now consider the second claim. In this case, for any time $t$ such that $|t|< |\tau_e|< T_\Lambda$ we can repeat the same arguments above and find $I(t)\in\mathbf{D}^\rho_{\Lambda,\frac34 r_\Lambda}(I(0))$. Then, by construction, the escape time satisfies
		\begin{equation}\label{chiusura}
		I(\tau_e)\in\text{closure}\big(\mathbf{D}^\rho_{\Lambda, \frac34 r_\Lambda}(I(0))\big) .
		\end{equation} Again, by Lemma \ref{GCB1}, this implies $|I(t)-I(0)|_2< \rho(\eps)$ for any $|t|<\tau_e< T_\Lambda$, so that, since $I(0)\in B_2\big(I(0),\tR(\eps)-(j+1)\rho(\eps)\big)$ one has
		\begin{equation}\label{palla}
		I(\tau_e)\in B_2\big(I(0),\tR(\eps)-j\rho(\eps)\big)\,.
		\end{equation} 
		
		Now, we shall prove that $I(\tau_e)\not\in Z_\Lambda$. By definition we have $I(\tau_e)\not\in D^\rho_{\Lambda,r_\Lambda}$ and, thanks to \eqref{blocco_esteso}, this means that there does not exist any action $ I^*\in D_\Lambda\cap B\big(I_0,\tR(\eps)-\rho(\eps)\big)$ such that $I(\tau_e)$ belongs to its disc $\mathbf{D}^\rho_{\Lambda,r_\Lambda}(I^*)$. Hence, by \eqref{disc}, $I(\tau_e)$ must satisfy at least one of the three following conditions:
		\begin{enumerate}
			\item $\not\exists I^*\in D_\Lambda\cap B_2\big(I_0,\tR(\eps)-\rho(\eps)\big):\, I(\tau_e)\in \bigcup_{I'\in I^*+\langle\Lambda\rangle}B_2(I',r_\Lambda)$;
			\item $I(\tau_e)\not\in Z_\Lambda$;
			\item  $I(\tau_e)\not\in B_2\big(I_0,\tR(\eps)-\rho(\eps)\big)$.
		\end{enumerate}
		By taking \eqref{chiusura} and \eqref{palla} into account, we see that the first and the third possibility cannot occur. Therefore, there must exist a maximal lattice $\Lambda'\neq\Lambda$ and a resonant zone $Z_{\Lambda'}$ such that $I(\tau_e)\in Z_{\Lambda'}$. Moreover,  Lemma \ref{GCB3}, insures that  $\dim\Lambda'\neq \dim\Lambda$ so that $I(\tau_e)\not\in Z_j$. The second decomposition in \eqref{covering} together with \eqref{palla} and \eqref{rhoerre} implies that $I(\tau_e)$ belongs to a resonant block of lower multiplicity, hence the claim. 
	\end{proof}
	\begin{rmk}
		The decompositions in \eqref{covering} are a covering of $B(I_0,\tR(\eps))$ but they are not a partition since, in general, $D_i\cap D_j\neq \varnothing $ for $j>i+1$. Hence, nothing prevents $I(\tau_e)$ from belonging to a resonant block of strictly higher multiplicity than the starting one. If this happens, however, thanks to the construction in \eqref{covering}, one is insured that $I(\tau_e)$ will also belong to another block associated to a lower order resonance. One therefore chooses the block in which to study the evolution of the actions once they leave the resonant zone they started at. This is at the core of the {\it resonant trap argument}, which is discussed in the sequel. 
	\end{rmk}
	\begin{proof}[Proof of Theorem \ref{MainTh3}]
		Theorem \ref{MainTh3} follows from Lemmas \ref{nr_stab_steep} and \ref{r_stab_steep}. Indeed, for any initial condition in the action variables $I_0\in B_\infty(0,R/4)$, we consider the ball $B_2(I_0,\tR(\eps))$ and the following dichotomy holds:
		\begin{enumerate}
			\item either $I_0$ belongs to the completely non-resonant domain $D_0^\rho$, in which case the proof ends here thanks to Lemma  \ref{nr_stab_steep};
			\item or for some $j\in\{1,...,n-1\}$ and some maximal $\Lambda\subset\Z^n_K$ of rank $j$, $I_0\in D_\Lambda\cap B\big(I_0,\tR(\eps)-(j+1)\rho(\eps)\big)$.
		\end{enumerate}
		In the second case, Lemma \ref{r_stab_steep} applies and one has another dichotomy:
		\begin{enumerate}
			\item either  $|I(t)-I(0)|_2\ledot  \rho(\eps):\circeq\eps^b$ over a time $T_\Lambda$; in this case the Theorem is proven since, taking into account the fact that the analyticity width in Lemma \ref{nr_stab_steep} satisfies $r\circeq \eps^{1/2}$, one has 
			\begin{equation}
			{\bf T}(\eps):= T_0:\circeq\frac{1}{\ |(1+a\ell)\ln\eps|^{\ell-1}\  \eps^{a(\ell-1)+1/2}}\circeq \frac{r}{\ |(1+a\ell)\ln\eps|^{\ell-1}\  \eps^{a(\ell-1)+1}}\ledot  \frac{\mathtt{T}_j\times r_\Lambda}{|\ln\eps^{6(1+a\ell)}|^{\ell-1}\eps^{a(\ell-1)+1}}\circeq :T_\Lambda\ ,
			\end{equation}
			where the last inequality is a consequence of the fact that, by \eqref{scegli}, \eqref{bidibibodibibu}, one can write
			$$
			r\le r_\Lambda \quad \longleftrightarrow \quad \frac{1}{K^{1+q_1}}\le \frac{1}{|\Lambda|K^{q_j}}
			$$
			and that, since $|\Lambda|\le K^j$, the stricter inequality
			$$ \frac{1}{K^{1+q_1}}\le \frac{1}{K^{j+q_j}}\quad\longleftrightarrow \quad 1+q_1 \ge j+q_j \quad \longleftrightarrow \quad p_1\ge p_j\ ,
			$$
			is trivially satisfied by the definition of $p_1$ and $p_j$, $j\in\{1,...,n-1\}$, in \eqref{parametrini} and by the fact that the steepness indices are always greater or equal than one.
			\item or the actions enter a resonant block $D_i\cap \bigg(B\big(I_0,\tR(\eps)-j\rho(\eps)\big)\bigg)$ corresponding to a resonant lattice of  dimension $i<j$ after having travelled a distance $\rho(\eps)$ over a time inferior to the time of escape. In this block, the above arguments can be repeated so that, after having possibly visited at most $n-1$ blocks, overall the actions can travel at most a distance $(n-1)\rho(\eps)$ before entering the completely non-resonant block, in which they are trapped for a time $T_0$ given by Lemma \ref{nr_stab_steep} and they travel for another length $\rho(\eps)$. Thanks to \eqref{rhoerre}, by construction one has  $|I(t)-I(0)|\le n\rho(\eps)=\frac12 \tR(\eps)\circeq \eps^b$.
		\end{enumerate}
		This is the so-called {\it resonant trap argument} and concludes the proof of Theorem \ref{MainTh3}, once one sets
		$$
		\mathtt{a}=a(\ell-1)+\frac12\quad,\qquad \mathtt{b}=b\ .
		$$
		
	\end{proof}

	\appendix
	\section{Smoothing estimates}\label{Fourier}
	
	\begin{lemma}
		\label{gigi}
		The derivatives of $K$ satisfy 
		\[\forall p\in \N, \quad \exists C_p\,: \abs{\partial^\beta K(x)} \le C_p \frac{e^{|\Im x|}}{(1 + |x|_2)^p},   \, \, \forall \,|\beta| \le p.  \]
		
	\end{lemma}
	For the proof see {\cite[Lemma $9$]{Chierchia_2003}}.
	\begin{lemma}\label{F1}
		Let $f\in C_b^\ell(\A^n)$, with $\ell\geq 1$, and let $\sum_{k\in\Z^n}\hat f_k(I)e^{ik\cdot \teta}$ be its Fourier series. Then, for any fixed $k\in\Z^n\backslash \{0\}$, there exists a uniform constant $\jeanpierre(n,\ell)$ satisfying
		\begin{equation}
		\norma{\hat f_k}_{C^0(\R^n)}\le \jeanpierre(n,\ell) \frac{\norma{f}_{C^q(\A^n)}}{|k|^q}\ ,
		\end{equation}
		where $q:=\floor{\ell}$.
	\end{lemma}
	\begin{proof}
		Fix a multi-index $j=(j_1,...,j_n)\in\N^n$ such that $|j|_1 \le q:=\floor{\ell}$, one obviously has
		\begin{equation}
		\partial^j_\teta f(I,\teta)=\sum_{k\in\Z^n}(i)^{|j|} k_1^{j_1}...k_n^{j_n}\hat f_k(I)e^{ik\cdot\teta}\,.
		\end{equation}	
		From
		\begin{equation}
		\partial^j_\teta f(I,\teta):=\sum_{k\in\Z^n}(\widehat{\partial^jf})_k(I)e^{ik\cdot\teta}\ ,
		\end{equation}
		and by the unicity of Fourier's coefficients one also has
		
		\begin{equation}\label{fourier_coef}
		\hat f_k(I):=\frac{(\widehat{\partial^jf})_k(I)}{k_1^{j_1}...k_n^{j_n}}\ .
		\end{equation}
		As in expression \eqref{fourier_coef} the multi-index $j\in \Z^n$ is arbitrary, for each value of $k\in\Z^n\backslash\{0\}$ we can choose $j$ so that 
		\begin{equation}\label{c5}
		\hat f_k(I) =\frac{(\widehat{\partial^jf})_k(I)}{(\max_{i=1,...,n}\{k_i\})^{|j|}}\ .
		\end{equation}
		Moreover, for any $k\in\Z^n\backslash\{0\}$ one has the trivial inequality $$
		\max_{i=1,...,n}\{|k_i|\}\geq \frac{|k|}{n}\ .
		$$
		This, together with \eqref{c5} and the choice $|j|=q$ yields{ 
			\begin{equation}
			|\hat f_k(I)| =n^\ell\frac{|(\widehat{\partial^jf})_k(I)|}{|k|^{q}}=n^\ell\frac{1/(2\pi)^n|\int_0^{2\pi}\partial^j f(I,\teta)e^{\im k\cdot\teta}d\teta|}{|k|^{q}}\le n^\ell\frac{|\partial^j f(I,\teta)|}{|k|^{q}}\ ,
			\end{equation}}
		which, once the supremum over the actions is taken, implies the result.
	\end{proof}

	\section{Normal form}\label{loch_app}
	Given a function $F$ in $\ciccio{r,s}$, the notations $\mathcal{P}_\Lambda$ and $\mathcal{P}_K$ stand for the projections
	$$
	\mathcal{P}_\Lambda F(I,\teta):=\sum_{k\in\mathbb{Z}^n:k\in\Lambda}F_k(I)e^{ik\cdot\teta}\ ,\ \ \mathcal{P}_K F(I,\teta) :=\sum_{k\in\mathbb{Z}^n:|k|_1\leq K}F_k(I)e^{ik\cdot\teta}\,
	$$
	
	Accordingly with our notations, we state here the result of P\"oschel \cite{Poschel_1993}.
	\begin{lemma}[Poschel's normal form]\label{poschel normal form}
		Let $\varrho,\sigma >0$ and $ {\bf H} (I,\teta)={\bf h}(I)+{\bf f}(I,\teta)$ be analytic on 
		$$
		\mathcal{D}_{\Lambda,\varrho,\sigma}:=\{(I,\teta)\in\C^n: \abs{I-D_\Lambda}_2<\varrho\ ,\ \  \teta\in\T^n_\sigma\}
		$$
		where $D_\Lambda$ is $(\alpha,K)$-nonresonant modulo $\Lambda$ with respect to the integrable Hamiltonian $\bf h$. 
		Also, let $M$ denote the hermitian norm of the hessian of ${\bf h}$ over $\mathcal{D}_{\Lambda,\varrho,\sigma}$. 
		
		If, for some $\varrho'>0$, one is insured
		\begin{equation}\label{soglie_Poschel}
		||{\bf{f}}||_{\varrho,\sigma}\le \epsilon \le \frac{1}{256\xi }\frac{\alpha \varrho'}{K}, \qquad \varrho'\le\left(\varrho,\frac{\alpha}{2\xi MK}\right)
		\end{equation}
		for some  $\xi > 1$ and 
		\begin{equation}\label{soglie_Poschel2}
		K\sigma\ge 6,
		\end{equation}
		then there exists a real-analytic,  symplectic transformation $\Psi:\mathcal{D}_{\Lambda,\varrho'/2,\sigma/6}\longrightarrow \mathcal{D}_{\Lambda,\varrho,\sigma}$ taking $\bf H$ into resonant normal form, that is 
		\begin{equation}
		{\bf H}\circ\Psi={\bf h}+{\bf g}+{\bf f^*}\ ,\ \ \{{\bf h},{\bf g}\}=0\ .
		\end{equation}
		Moreover,
		denoting by ${\bf g_0}:=P_\Lambda P_K{\bf f}$ the resonant part of $\bf f$, we have the estimates 
		\begin{equation}
		||{\bf g-g_0}||_{\varrho'/2,\sigma/6}\le 64 \frac{K}{\alpha\varrho' }\epsilon^2\ ,\ \  ||{\bf f^*}||_{\varrho'/2,\sigma/6}\le e^{-K\sigma/6}\epsilon.
		\end{equation}
		Furthermore, $\Psi$ is close to the identity, in the sense that, for any $(I,\teta)\in \mathcal{D}_{\Lambda,\varrho'/2,\sigma/6}$, one has
		\begin{equation}\label{taglia_trasf}
		\frac{\abs{\Pi_I \Psi-I}_2}{\rho'}\le 2^3 \frac{K}{\alpha\rho'}\epsilon\le \frac{1}{32\xi}\quad ,\qquad \frac{\abs{\Pi_\teta \Psi-\teta}_\infty}{\sigma}\le \frac{2^5K}{3\alpha\rho' }\epsilon \le \frac{1}{24\xi} 
		\end{equation}
		where $\Pi_I,\Pi_\teta$ denote the projection on the action and angle variables, respectively.
	\end{lemma}
	
\gr{Declarations}. Data sharing not applicable to this article as no datasets were generated or analyzed during the current study.

\noindent
Conflicts of interest: The authors have no conflicts of interest to declare.

\end{document}